\documentclass[10pt,twoside]{article}
\usepackage{amsfonts}
\usepackage{amscd}
\usepackage{amsmath}
\usepackage{verbatim}
\usepackage{amssymb}
\usepackage{indentfirst}
\newtheorem{theorem}{Theorem}
\newtheorem{theorema}{Theorem}
\newtheorem{theoremb}{Theorem}

\newtheorem{cor}[theorem]{Corollary}

\newtheorem{dfn}[theorema]{Definition}

\newtheorem{lem}[theorem]{Lemma}

\newtheorem{prop}[theorem]{Proposition}
\newtheorem{remark}[theoremb]{Remark}

\newenvironment{proof}[1][Proof]{\textbf{#1.} }{\qed}

\newcommand\A{{\cal A}}
\renewcommand\a{\alpha}
\renewcommand\b{\beta}

\newcommand\C{{\cal C}}
\newcommand\CC{{\mathbb C}}

\renewcommand\d{\delta}
\newcommand\D{{\cal D}}

\newcommand\E{{\cal E}}

\newcommand\g{{\frak g}}

\renewcommand\ll{\lambda}
\newcommand\La{\Lambda}

\newcommand\oo{\omega}
\newcommand\op[1]{\mathop{\rm #1}\nolimits}
\newcommand\ot{\otimes}
\newcommand\p{\partial}

\newcommand\R{{\mathbb R}}

\renewcommand\t{\tau}

\newcommand\T{\Theta}

\newcommand\ve{\varepsilon}

\newcommand\vp{\varphi}

\newcommand\we{\wedge}

\newcommand\z{\sigma}
\newcommand\Z{{\mathbb Z}}

\def\Rom#1{\uppercase\expandafter{\romannumeral#1}}

\newcommand\qed{\phantom{\underline{y}}\hfill\hfill$\square$}

\newcommand\bib[1]{\bibitem[#1]{#1}}

\renewcommand\mod{\hskip2.5pt\op{mod}}

\makeatletter
\renewcommand{\@evenhead}{\hfil Boris Kruglikov \ $\diamond$ \ Valentin Lychagin\hfil}
\renewcommand{\@oddhead}{\hfil Global Lie-Tresse theorem\hfil}
\renewcommand{\@begintheorem}[2]{\begin{trivlist}\it
 \item[\hspace{\labelsep}{\bf #1\ #2.}]}
\renewcommand{\@endtheorem}{\end{trivlist}}
\def\@seccntformat#1{\csname the#1\endcsname.\quad}
\def\numberline#1{\hb@xt@\@tempdima{#1\if&#1&\else.\fi\hfil}}
\makeatother

\begin{document}

\title{\bf Global Lie-Tresse theorem}
\author{\textsc{Boris Kruglikov and Valentin Lychagin}}
\date{}

\maketitle

\begin{abstract}
We prove a global algebraic version of the Lie-Tresse theorem which states that the
algebra of differential invariants of an algebraic pseudogroup action on a
differential equation is generated by a finite number of rational-polynomial
differential invariants and invariant derivations.
 \footnote{
algebraic group, pseudogroup action, rational invariant, differential invariant,
invariant derivation, Tresse derivative, differential syzygy, separation of orbits, Spencer cohomology.}%
\end{abstract}

\section*{Introduction}\label{sec0}

According to the Erlangen program of F.\,Klein a geometry is characterized by the invariants
of a transitive transformation group. Finite generation property for
algebraic invariants was the topic of D.\,Hilbert's XIV problem.

For infinite groups of S.\,Lie and E.\,Cartan (as well as for usual Lie groups)
\cite{Li$_1$,C$_1$} (see also \cite{SS,S$_1$,AM})
the Noetherian property generally does not hold for the algebra of differential invariants
(of arbitrary order), and instead finiteness is guaranteed by the Lie-Tresse theorem,
which uses invariant functions and invariant derivations as generators.

This theorem is a phenomenological statement motivated by Lie and Tresse
\cite{Li$_2$,Tr$_1$}. It was rigorously proved for un-constrained actions of pseudogroups
(i.e.\ on regular strata in the space of jets) in \cite{Kum}, see also \cite{Ov,Ol$_1$,MMR,SSh}.
This was generalized for pseudogroup actions on differential equations
(which can be, for instance, singular strata of un-constrained actions) in \cite{KL$_2$}.
Algorithmic construction of differential invariants via Gr\"obner basis technique
(in the case of free actions) is done in \cite{OP}.

In all these approaches the regularity issue plays a significant role and
the generating property of the algebra $\A$ of differential invariants
(as well as the very definition of $\A$) holds {\em micro-locally\/},
i.e.\ on open domains in the space of infinite jets $J^\infty$ (usually these
domains are not Zariski open or $G$-invariant, and their jet-order and size
depend on the element of $\A$).

In the present paper we overcome this difficulty by considering {\it algebraic actions\/}
of Lie pseudogroups (essentially all known examples are such) and restricting to differential
invariants that are rational functions in jets of certain order $\le l$ and polynomial in higher jets.
These invariants will separate the regular orbits and will be finitely generated in the Lie-Tresse sense globally.

The set of singularities is algebraic in $J^\infty$, i.e.\ given by a finite number
of finite jet-order relations, specifying when the generators are not defined (denominators vanish)
or are dependent (some determinants vanish). This is one of our results on {\it stabilization
of singularities\/}. Then we prove the Lie-Tresse theorem in the complement to these singularities.
This global result is an enhancement of the micro-local smooth situation\footnote{In the local case
one restricts to a neighborhood $\hat U=\pi_{\infty,0}^{-1}(U)\subset J^\infty$ for an open set $U\subset M$,
while in the micro-local case $\hat U=\pi_{\infty,k}^{-1}(U)\subset J^\infty$
for some open set $U\subset J^k$.}
since the latter can be easily deduced by an application of the implicit function theorem.

{\it Rational differential invariants\/} are natural in the classification problems, and all
known examples are such (see the discussion in \S \ref{S43} about appearance of roots).
Our paper justifies this experimental observation.

It is important to stress that the separation property holds globally and this gives
a method to distinguish between different regular orbits of the pseudogroup actions.
This implies, in particular, a possibility to solve algebraic equivalence problems
via the differential-geometric technique (see e.g.\ \cite{BL}).

\subsection{Algebraic pseudogroups}\label{S01}

A {\em Lie pseudogroup\/} $G\subset\op{Diff}_{\op{loc}}(M)$ acting on a
manifold $M$ consists of a collection of local diffeomorphisms
$\vp$, each bearing its own domain of definition $\op{dom}(\vp)$ and
range $\op{im}(\vp)$, that satisfies the following properties:
 \begin{enumerate}
  \item $\op{id}_M\in G$ and $\op{dom}(\op{id}_M)=\op{im}(\op{id}_M)=M$,
  \item If $\vp,\psi\in G$, then $\vp\circ\psi\in G$ whenever
$\op{dom}(\vp)\subset\op{im}(\psi)$,
  \item If $\vp\in G$, then $\vp^{-1}\in G$ and $\op{dom}(\vp^{-1})=\op{im}(\vp)$,
  \item $\vp\in G$ iff for every open subset $U\subset\op{dom}(\vp)$
the restriction $\left.\vp\right|_U\in G$,
  \item The pseudogroup is of order $l$ if this is the minimal number such
that $\vp\in G$ whenever for each point $a\in\op{dom}(\vp)$ the
$l$-jet is admissible: $[\vp]^l_a\in G^l$.
 \end{enumerate}

Property 5 means that the pseudogroup is defined by a
{\em Lie equation\/} of differential order $l$, i.e.\ the embedding
$G^l\subset J^l(M,M)$ uniquely determines the higher groupoids $\{G^k\}_{k\ge l}$ (defined below)
by the prolongation technique, and it even determines $G$.
Analytic pseudogroups transitive on $M$ are always Lie pseudogroups as
follows from the Cartan-Kuranishi theorem (see also Remark \ref{rem2} in \S \ref{S12}).
In most part of the text below we will simply write "pseudogroup".

In this paper we assume throughout that the action of $G$ on $M$ is {\it transitive\/},
i.e.\ any two points on $M$ can be superposed by a pseudogroup element.
In this case property 5 in the above list can be removed from the assumptions,
without altering any of our results (we however do not use this).
In the {\it intransitive case\/} some of our results still hold under additional assumptions.

A transformation $\vp\in G$ defines a map ($l$-th prolongation) of
the space of jets of dimension $n$ submanifolds
$\vp^{(l)}:J^l(M,n)\to J^l(M,n)$, which obeys the following
property:
 $$
(\vp\circ\psi^{-1})^{(l)}=\vp^{(l)}\circ(\psi^{(l)})^{-1}.
 $$
This action on $J^\infty(M,n)$ induces the action of $G$ on the space of functions
$\mathcal{F}(J^\infty(M,n))$, which is the main object of our study.
The alternatives for $\mathcal{F}$ are:
smooth or analytic functions, rational functions or polynomials with respect to the
coordinates on the fibers of the bundle $J^\infty\to J^1$ (fibers of the bundle
$\pi_{i+1,i}:J^{i+1}\to J^i$ have a natural affine structure for $i\geq1$).
We will use a (reasonably minimal) combination of these.

Let us denote the stabilizer of order $k$ of the point $a\in M$ by $G^k_a=\{\vp\in G^k:\vp(a)=a\}$
(in the transitive case dependence on $a$ is not essential: these sub-groups are conjugate).
The group $G^k_a$ acts on the $k$-jets of dimension $n$ submanifolds
at $a$, which is an algebraic manifold $J_a^k(M,n)$, see \S\ref{S1} for details.

 \begin{dfn}
The $G$ action on $M$ is called algebraic if for the order $l$ of the pseudogroup
the stabilizer $G^l_a$ is an algebraic group acting algebraically on $J^l_a$.
 \end{dfn}
Actually we require that $G^l_a$ is an algebraic subgroup of the differential group $D^l_a$
(see \S\ref{S12} for more details).
It is then straightforward to check that the prolonged action of $G^k_a$ on $J^k$ is also
algebraic for all $k\ge l$.

We want more than this, namely we wish to consider a differential equation $\E$ as a submanifold in jets,
such that $\E^k\subset J^k$ is $G$-invariant for all $k$ (in the general setup, for which
our results are valid, the action of $G$ is defined only on $\E$ and is not induced by an ambient action).

In this case we have an action of $G^k_a$ on $\E^k_a=\pi_k^{-1}(a)$ and we call it {\it algebraic\/}
if $\E^k_a\subset J^k_a$ for any point $a\in M$ is an algebraic (nonsingular) submanifold on which
$G_a^k$ acts algebraically. We call $\E$ {\em algebraic differential equation\/} in this case.
Notice that this property concerns the behavior only with respect
to the derivatives, so for instance sin-Gordon $u_{xy}=\sin u$ or Liouville $u_{xy}=e^u$
are algebraic differential equations from this perspective.

Without loss of generality we can assume that the maximal order of $\E$ is $\le l$.
Then the assumption that $\E^l$ is algebraic and {\it irreducible\/}
implies the same property for the prolongations $\E^k=(\E^l)^{(k-l)}$ provided the system is
formally integrable (this is relevant for overdetermined systems only, see \S \ref{S11}).
These will be the standing assumptions of this paper.

\subsection{Main results and discussion}\label{S02}

Like in the classical invariant theory, the pseudogroup actions possess
the algebra $\A$ of differential invariants, which are simply the invariant functions
of the prolonged action in jets $\vp^{(k)}\in\op{Diff}_\text{loc}(J^k)$, $\vp\in G$, for all $k$.
Generally the number of independent differential invariants is infinite.
But micro-locally $\A$ can be finitely generated via the Lie approach with a finite
number of invariant derivations and a finite number of differential invariants,
or via the Tresse method using differentiation of some invariants by the others (see \S\ref{S14}).

We call a closed subset $S\subset J^k$ {\it Zariski closed\/} if its intersection $S_a$ with every
fiber $J^k_a$, $a\in M$, is Zariski closed. The same concerns subsets $S\subset\E^k$.
If $S$ is $G$-invariant, it projects onto $M$ (the action is assumed transitive).
Formally integrable equations correspond to compatible systems and are discussed in \S1.

Our first result is the analog of the Cartan-Kuranishi theorem
(that regularity is guaranteed by a finite number of conditions) for pseudogroup actions.

 \begin{theorem}\label{Thm1}
Consider an algebraic action of a pseudogroup $G$ on a formally integrable irreducible
differential equation $\E$ over $M$. Suppose $G$ acts transitively on $M$.
Then there exists a number $l$ and a Zariski closed invariant proper subset $S_l\subset\E^l$ such that
the action is regular in $\pi_{\infty,l}^{-1}(\E^l\setminus S_l)\subset\E^\infty$, i.e.\
$\E^k\setminus\pi_{k,l}^{-1}(S_l)$  for any $k\ge l$ is algebraically fibered by
the orbits of $G^k$ (in particular the orbits are closed and have the same dimension).
In other words, there exists a rational geometric quotient
 $$
\bigl(\E^k\setminus\pi_{k,l}^{-1}(S_l)\bigr)/G^k\simeq Y^k.
 $$
 \end{theorem}

Consequently, we have the following stabilization of singularities: By Rosenlicht's theorem \cite{R$_1$}
(see \S \ref{S21}) the action on $\E^k$ has a geometric quotient outside a proper set $S_k$ of singularities
for all $k>0$. Then for some number $l$ (which can exceed the order of the pseudogroup $G$), we have:
$\pi_{k,l}(S_k)\subset S_l$ for $k\ge l$.

\smallskip

Denote by $\mathfrak{P}_l(\E)$ the algebra of functions on $\E^\infty$
(this means $\cup_k\pi_{\infty,k}^*(\mathcal{F}(\E^k))$)
that are smooth in the base variables, rational in the fiber variables of $\pi_l:J^l\to M$
and polynomial in the higher jets of order $k>l$.\footnote{Ultimately, when $l\to\infty$, we get
functions that are smooth in the base variables and rational on the fibers.
The theory applies to this class as well (and is simpler).}
We will look for differential invariants among such functions, which we call {\it rational-polynomial}.

Our second main result is the global Lie-Tresse theorem for pseudogroup actions on differential equations
(it includes the important case of un-constrained action, when $\E^l=J^l$ is the entire jet-space).

 \begin{theorem}\label{Thm2}
With the assumptions of the previous theorem, there exists a number $l$ and a Zariski closed
invariant proper subset $S\subset\E^l$ such that the algebra $\mathfrak{P}_l(\E)^G$ of differential
invariants separates the regular orbits from $\E^\infty\setminus\pi_{\infty,l}^{-1}(S)$ and
is finitely generated in the following sense.

There exists a finite number of functions $I_1,\dots,I_t\in\mathfrak{P}_l(\E)^G$ and a finite number
of rational invariant derivations $\nabla_1,\dots,\nabla_s:\mathfrak{P}_l(\E)^G\to \mathfrak{P}_l(\E)^G$
such that any function from $\mathfrak{P}_l(\E)^G$ is a polynomial of differential
invariants $\nabla_{J}I_i$, where $\nabla_J=\nabla_{j_1}\cdots\nabla_{j_r}$
for some multi-indices $J=(j_1,\dots,j_r)$, with coefficients being
rational functions of $I_i$.
 \end{theorem}

The number $l$ might not be the same as in Theorem \ref{Thm1} and is usually greater than the stabilization
jet-level of $G$ and $\E$. It is determined by the stabilization in cohomology as shown in Section \ref{S2}.
The set $S$ includes the previous singularity locus $S_l$ but it can be bigger (for instance, it can include
points where the rank of differentials of the generating invariants $\nabla_JI_i$ of order $\le l$ drop).
The number $t$ can be taken equal to the codimension of the regular orbit in $\E^l$ plus 1.
The number $s$ does not exceed the dimension of the affine complex characteristic variety of $\E$
(this notion is defined in Section \ref{S1}).

Let us remark that the derivations $\nabla_j$ not always can be taken as Tresse derivatives/horizontal
vector fields --- in Section \ref{S4} we show examples, where growth of the dimension of the algebra
of differential invariants on $\E^k$ is a polynomial in $k$ of degree $<n$,
and the number of invariant derivations is $s<n$.

In the above approach we remove not only the genuine singular orbits but also some
regular ones. We can minimize the amount of removed
regular orbits (shrink $S$) at the price of increasing the number of basic invariants $I_i$.

  \begin{remark}
It is not enough to consider only polynomial invariants, as readily follows from Nagata's counter-example
to Hilbert XIV: the algebra of polynomial invariants needs not be finitely-generated  \cite{PV}.
Hilbert's finiteness property holds for reductive groups, but the sub-groups $G^k_x$ arising in the
pseudogroup actions are seldom such. In addition, the polynomial invariants often do not separate
the regular orbits, that's why we need rational invariants
(however regular orbits can be separated by polynomial relative differential invariants).
  \end{remark}

Some other results proved in this paper are: finiteness theorems for
invariant derivations, differential forms and other natural geometric
objects (tensors, differential operators, connections etc). We also prove
finite generation property for differential syzygies and higher syzygies,
interpreting this as a $G$-invariant Cartan-Kuranishi theorem.

\smallskip

Let us remark that the results of the paper hold true for some other classes of functions.
Namely if the manifold $\E^l$ and the $G$-action is real or complex analytic,
then we can consider the field $\mathfrak{R}(\E)$ consisting of functions that are analytic
in the base variables and rational on the fibers of $\pi_\infty:J^\infty\to M$.
Our theory of differential invariants applies to this class as well, and we can conclude
finiteness and separation property.
Similarly, it holds if the action of $G$ is transitive on $\E^k$ and is algebraic on the
fibers of $\E^\infty\to\E^k$.

If $\E^l$ is a complex analytic Stein space (for $l\gg1$) and $G$ a complex-reductive Lie group
acting upon it by symmetries, then results of \cite{GM} imply existence of rational quotient
on a finite jet-level. Then the results of Section \ref{S3} yield a global structure theorem
in the class of meromorphic functions $\mathcal{M}(\E)$, i.e.\ meromorphic rational differential
invariants are finitely generated in the Lie-Tresse sense and they separate the regular orbits.

The classes $\mathfrak{P}_l(\E)$, $\mathfrak{R}(\E)$ and $\mathcal{M}(\E)$ are sufficiently rich,
yet controllable. Differential invariants in the larger space of smooth functions $C^\infty(\E)$
can fail to satisfy the finiteness property of the generalized Lie-Tresse theorem globally.

\subsection{Overview of the problem and Outline of the paper}\label{S03}

Differential invariants play an important role in the solution of the equivalence problem
for geometric structures (K\"ahler and Einstein metrics, projective and conformal structures,
webs  etc) and integration of differential equations.

Their investigations have origin in the works by Lie, Halphen, Tresse and Cartan,
with later advances by Sternberg, Kobayashi, Chern and Tanaka to name a few.
This includes the theory of representations and the structure of infinite pseudogroups
\cite{C$_1$,SS,GS$_1$,Kum,RK}.

Structure description of the algebra of differential invariants, in spite of the recent
interest and progress, was mainly micro-local or was limited to locally free and
un-constrained actions (i.e.\ action on the whole space of jets), see \cite{Th,KJ,KL$_4$,OP,Man}.

Many classical irrational algebraic expressions (containing roots) are local or even micro-local
but not global invariants; we claim that all global $G$-invariants are generated by
rational differential invariants.

To see this consider the well-known formula for the curvature of curves, which illustrates the
difference between micro-local and global approaches.

\medskip

{\bf Example.} The proper motion group $E(2)_+=SO(2)\ltimes\R^2$ acts on the curves in Euclidean
$\R^2(x,y)$. The classical curvature
 $$
\kappa=\frac{y''(x)}{\sqrt{(1+y'(x)^2)^{3}}}
 $$
is not an invariant of $E(2)_+$ (indeed, the reflection $(x,y)\mapsto(-x,-y)$ preserves the circle
$x^2+y^2=r^2$ but transforms $\kappa\mapsto-\kappa$; $\kappa=\pm r^{-1}$), but its square
$\kappa^2$ is a bona fide rational differential invariant.

Notice however that the Lie group $E(2)_+$ is connected, and that $\kappa$ is invariant under the action of
its Lie algebra (resp.\ local Lie group). What happens is that under the lifted action of $S^1$ on the space
of 1st jets, the derivative $y'(x)$ becomes infinite and changes the branch, whence the change of sign.

\medskip

In this paper we show that this is the common situation with the algebraic actions and establish
the structure theorem for the algebra of differential invariants (with Lie-Tresse generating property).
The specification of behavior of differential invariants (rational-polynomial) makes it possible to
obtain the first global result on the subject.

\medskip

Structure of the paper is as follows. In Section \ref{S1} we discuss the basics from the
geometric theory of differential equations, including symbolic modules, characteristics,
Spencer cohomology and prolongations, and we also introduce pseudogroups, actions,
and define differential invariants and invariant derivations. Most of these notions are
standard, but we make some important specifications for the needs of the present paper.

In Section \ref{S2} we introduce our main tools: symbolic calculus of the differential invariants,
together with the corresponding Spencer-like cohomology that count the differential invariants.
Working in the algebraic category we recall Rosenlicht's theorem and the notion of
quotient by the action. Then we prove stabilization (a Noetherian property) for this symbolic system,
and deduce stabilization of cohomology. We also establish existence of invariant derivations
(under no additional assumptions like those of \cite{Kum}).

In Section \ref{S3} we combine Rosenlicht's theorem with the results from Section~\ref{S2}
to obtain our main results. We also generalize the results to obtain finiteness of other geometric
quantities for the algebraic actions of pseudogroups. In particular, we prove finite generation
property of the module of invariant derivatives, and indicate, when we can choose $n$
independent invariant derivatives, in which case they coincide with invariant derivations.
We briefly indicate how to overcome the restriction that $G$ is algebraic: essentially
one needs to separate only the orbits in low level jets because the prolonged actions are
always algebraic in higher jets.
Then we discuss differential syzygy and their finiteness.
Finally we bound the growth of the Hilbert function. This implies rationality
of the Poincar\'e series, solving the problem of V.\,Arnold on counting the number of moduli
of geometric structures with respect to transitive pseudogroup actions.

An example of a pseudogroup action presented at the end of Sections \ref{S2} and \ref{S3}
illustrates all the introduced notions and claims.

In Section \ref{S4} we perform some calculations. We start by recalling the basic methods
to compute differential invariants, and then we present some new examples. They are chosen
to demonstrate importance and sharpness of our main assumptions.
In particular, we show that dropping the assumption of algebraic action or
considering non-transitive actions can lead to violation of separation or finite generation
property in the Lie-Tresse theorem (both local and global versions). The examples are also of
independent interest as they are related to classical geometric problems.

Notations in Section \ref{S1} will be used generously, so that $\A$ can denote
$\mathfrak{P}_l(\E)^G$, $\mathfrak{R}(\E)^G$, $\mathcal{M}(\E)^G$ or $C^\infty(\E)^G$
depending on the context. For definiteness we address mostly the case of
smooth functions (globally or micro-locally), the other definitions being analogous.
Starting from Section \ref{S2} however we restrict $\A$ to be our main algebra of
rational-polynomial invariants (if not specified otherwise) and the same concerns derivations
and other invariant objects.

\section{Pseudogroup actions}\label{S1}

In this section we discuss the general introduction of pseudogroups,
differential equations and actions developing the ideas of \cite{GS$_2$,SS,S$_2$}.
In some parts the exposition follows \cite{KL$_4$}.

\subsection{Jets and the geometric theory of PDEs}\label{S11}

For a smooth manifold $M$ of dimension $m$ denote by $J^k(M,n)$ the space of $k$-jets
$a_k=[N]^k_a$ of $n$-dimensional submanifolds $N\subset M$ through $a$, taken over all points $a\in M$.
Every submanifold $N$ of dimension $n$ determines uniquely the jet-extension
$j_k(N)=\{[N]^k_a:a\in N\}\subset J^k(M,n)$.

We have $J^0(M,n)=M$ and $J^1(M,n)=\op{Gr}_{n}(TM)$.
The natural projections $\pi_{k,k-1}:J^k(M,n)\to J^{k-1}(M,n)$ allow us to define
the projective limit $J^\infty(M,n)=\lim J^k(M,n)$.

The fibers $F(a_{k-1})=\pi_{k,k-1}^{-1}(a_{k-1})$ for $k>1$ carry a canonical
affine structure \cite{Go,KLV}, associated with the vector structure described below.

Denote $\t=T_aN=[N]^1_a$ and $\nu=T_aM/T_aN$ (both are determined by $a_1$).
Let $a_k\in J^k(M,n)$, $a_{k-1}=\pi_{k,k-1}(a_k)$.
Then $T_{a_k}F(a_{k-1})\simeq S^k\t^*\ot\nu$ and we get the exact sequence:
 $$
0\to S^k\t^*\ot\nu\to T_{a_k}J^k(M,n)\stackrel{(\pi_{k,k-1})_*} \longrightarrow
T_{a_{k-1}}J^{k-1}(M,n)\to 0.
 $$
Thus the affine structure of $F(a_{k-1})$ is modeled on the vector space $S^k\t^*\ot\nu$.

If $M$ is the total space of a vector bundle $\pi:M\to B$ of rank $m-n$, the space
$J^k\pi\subset J^k(M,n)$ is an open dense subset consisting of jets of
local sections, i.e.\ submanifolds $N$ transversal to the fibers of $\pi$.

A system of differential equations (PDE) of maximal order $l$ is a
sequence $\E=\{\E^k\}_{0\le k\le l}$ of submanifolds $\E^k\subset J^k(M,n)$ with $\E^0=M$
such that for all $0\le k<l$ the following conditions hold:

 \begin{enumerate}
  \item
$\pi^\E_{k+1,k}:\E^{k+1}\to\E^k$ are smooth fiber bundles.
  \item
Denote $L(a_{k+1})=T_{a_k}j_k(N)$ for $a_{k+1}=[N]_a^{k+1}$.
The first prolongations
 $$
(\E^k)^{(1)}=\{a_{k+1}\in J^{k+1}\,|\,L(a_{k+1})\subset T_{a_k}\E^k\}.
 $$
are smooth subbundles of $\pi_{k+1}$ and $\E^{k+1}\subset(\E^k)^{(1)}$.
 \end{enumerate}

The higher prolongations are defined either by requiring the higher contact of $j_k(N)$ with $\E^k$
or inductively. We let $\E^{i+l}=(\E^l)^{(i)}$ and assume that the projections $\pi_{k+1,k}:\E^{k+1}\to\E^{k}$
are (affine) bundles for $k\ge l$ --- this is equivalent to formal integrability (compatibility) of the system $\E$ \cite{S$_2$,Go}.

Consider a point $a_l\in\E^l$ with $a_i=\pi_{l,i}(a_l)$ for $i<l$.
It determines the symbolic system $\{g_i\}_{i=0}^\infty$ by the formula
 $$
g_i=T_{a_i}\E^i\cap F(a_{i-1})\subset S^i\t^*\ot\nu\ \text{ for }\ i\le l
 $$
and $g_i=g_l^{(i-l)}=g_l\ot S^{i-l}\t^*\cap S^i\t^*\ot\nu$ for $i>l$.
By the above conditions the symbols $g_i$ form smooth vector bundles over
$\E^l$ and $g_i\subset g_{i-1}^{(1)}$. This latter inequality alone determines
a {\it symbolic system\/} $g$, see \cite{KL$_1$}.

The Spencer $\d$-complex for the PDE system $\E$ at a point $a_l\in\E^l$
(or for the symbolic system $g$) is
 $$
\cdots\to g_{i+1}\ot\La^{j-1}\t^*\stackrel\d\longrightarrow
g_i\ot\La^j\t^*\stackrel\d\longrightarrow
g_{i-1}\ot\La^{j+1}\t^*\stackrel\d\to\cdots
 $$
where $\delta$ is the Spencer operator
given by $\delta(p^i\otimes v\otimes\omega)=ip^{i-1}\otimes v\otimes(p\we\omega)$ for
$p\in\t^*$, $v\in\nu$, $\omega\in\La^j\t^*$ (clearly $S^i\t^*\ot\nu$ is generated by the elements $p^i\ot v$).

The cohomology $H^{i,j}(\E;a_l)=H^{i,j}(g)$ of the above complex at the term $g_i\ot\La^j\t^*$
is called the {\em Spencer $\d$-cohomology group\/}.
The formal integrability of the system $\E$ is equivalent to the regularity condition
(the symbolic system $g$ is a bundle over $\E^l$) and to vanishing
of certain curvature type tensors $W_k(\E)\in H^{k-1,2}(\E)$, see \cite{Ly}
(in particular, if the system $\E$ is involutive, i.e.\ $H^{i,j}(\E)=0$ for $i\ge l$, and
$\pi_{l+1,l}:\E^{l+1}\to\E^l$ is onto, then $\E$ is formally integrable \cite{S$_2$,Go}).

In general some $W_k$ can be nonzero, and then one needs to perform the Cartan-Kuranishi completion to
involution. This usually results in shrinking of the equation $\E$ as a submanifold in $J^\infty(M,n)$
and removing singular strata. In our algebraic situation this
boils down to enlarge the proper Zariski closed subvariety $S_l$,
so we will suppose from the beginning that $\E$ is formally integrable.

\subsection{Lie pseudogroups}\label{S12}

The construction of the previous Section yields as a partial case the jet space
for maps $J^k(M,M)\subset J^k(M\times M,m)$. This contains, as an open dense
subset, the jet-space for diffeomorphisms $D^k(M)$ consisting of $k$-jets of $m$-dimensional
submanifolds $N\subset M\times M$, which diffeomorphically project to both factors.

This $D^k(M)$ equipped with the partially defined composition operation is
a {\em Lie groupoid\/}, basic for the definition of a general finite order pseudogroup.
Its stabilizer at $a$, $D_a^k$ is called the {\it differential group\/} of order $k$.
Let us remark that $D_a^k$ is an affine algebraic group because it is an automorphism
group of the space of $k$-jets with the fixed point.

 \begin{dfn}\label{pseudogr}
A Lie pseudogroup of order $l$ is given by a Lie equation, which is a collection of sub-bundles
$G^j\subset D^j(M)$, $0<j\le l$, such that the following properties are satisfied:
 \begin{enumerate}
  \item If $\vp_j,\psi_j\in G^j$, then
$\vp_j\circ\psi_j^{-1}\in G^j$ whenever defined,
  \item $G^j\subset(G^{j-1})^{(1)}$ and $\rho_{j,j-1}:G^j\to G^{j-1}$ is a bundle for every $j\le l$.
 \end{enumerate}
 \end{dfn}
We assume {\it transitivity\/} of the pseudogroup action, i.e.\ $G^0=D^0(M)=M\times M$.
Assumption 1 implies that $\op{id}^j_M\in G^j$ and
$\vp_j\in G^j\Rightarrow\vp_j^{-1}\in G^j$.

Pseudogroups $G=\{G^j\}$ defined by this approach, with $G^{l+j}=(G^l)^{(j)}$ for $j>0$, can be
studied for integrability by the standard prolongation-projection method
\cite{GS$_1$,GS$_2$,Kur,KLV,S$_1$}.
We will assume formal integrability from the beginning, so that $G^k$ is a bundle over $G^{k-1}$
for any $k>0$.

 \begin{remark}\label{rem2}
By Cartan's first fundamental theorem \cite{SS} every Lie algebra sheaf is a sheaf of
infinitesimal automorphisms of a transitive finite order structure. This is equivalent to
the claim that transitive finite order pseudogroups are identical with Lie pseudogroups.
 \end{remark}

Denote $G^j_a=\{\vp_j\in G^j\,|\,\vp_0(a)=a\}$ the isotropy subgroup of $G^j$.
Transitivity implies that the subgroups $G^j_a\subset D^j_a(M)$ are conjugate for
different points $a\in M$. By our assumption these $G^j_a$ are {\it algebraic subgroups\/}
of $D_a^j$ for $j\le l$. It follows that the prolongations $G^k_a\subset D^k_a(M)$ are also
algebraic for all $k>l$.


 \begin{remark}\label{rem3}
For our basic theorems (description of orbits and invariants in $\E$)
we only need the action of $G^k$, so they are true for general pseudogroups, and from now on we will
talk only of pseudogroups (instead of Lie pseudogroups).
 \end{remark}

 \begin{dfn}
The symbol of the pseudogroup $G$ at a point $\vp_j\in G^j_a$ is
 $$
\g^j(\vp_j)=\op{Ker}\bigl[(\rho_{j,j-1})_*:T_{\vp_j}G^j\to
T_{\vp_{j-1}}G^{j-1}\bigr]\subset S^j(T^*_aM)\ot T_aM.
 $$
 \end{dfn}

Notice that contrary to finite-dimensional Lie groups (which can be abstract),
the infinite pseudogroups always come equipped with their natural representations. In our setup
$G^k$ acts on $J^k(M,n)$ via the action of local diffeomorphisms on submanifolds
$\vp_k:[N]_a^k\mapsto[\vp(N)]_{\vp(a)}^k$.

An important example is given by the natural bundles $\pi:M\to B$, with $G$ being lifted
from the pseudogroup $\op{Diff}_\text{loc}(B)$ on the base. Here $J^\infty\pi$ is the bundle
of formal geometric structures \cite{ALV}
and the global Lie-Tresse theorem yields the framework for classification of differential
invariants for the geometric structures of a given type.

\subsection{Differential invariants and invariant derivations}\label{S13}

The equivalence problem is to decide when a submanifold $N_1\subset M$ can be transformed
to a submanifold $N_2\subset M$ by a map $\vp\in G$ (if an equation $\E$ is imposed,
both submanifolds $N_i$ must be solutions of it)\footnote{The equivalence problems of
geometric structures on a manifold can be set up to be a partial case of our general setting
for equivalence of submanifolds.}. A sub-problem is related to
equivalence of $\infty$-jets of submanifolds -- then in many cases (algebraic,
analytic, germs of submanifolds given by elliptic equations etc) the formal
equivalence implies the local and eventually the global one.

Thus we have to distinguish between $G$-orbits in $\E^\infty\subset J^\infty(M,n)$ and this is precisely what
the differential invariants (i.e.\ functions constant on these orbits) do. Consider the algebra of smooth
functions $C^\infty(\E^\infty)$, which is the inverse limit of the algebras $C^\infty(\E^k)$ via the maps $\pi_{k,k-1}^*$. In many cases it is convenient to work with the sheaf $C^\infty_\text{loc}(\E^k)$
of germs of such smooth functions.

 \begin{dfn}
A function $f\in C^\infty_\text{\rm{loc}}(\E^k)$ is a differential invariant of order $k$ if it is invariant
with respect to the $G$-action on $\E^k$. The algebra of  differential invariants of order $k$ is
denoted by $\mathcal{A}_k$, and $\mathcal{A}=\lim\mathcal{A}_k$ is the algebra of all smooth differential invariants.
 \end{dfn}

Thus the algebra of differential invariants is filtered $\mathcal{A}=\cup\mathcal{A}_k$ via the natural inclusions $\pi_{k+1,k}^*:\mathcal{A}_k\to\mathcal{A}_{k+1}$. In addition to the usual algebraic operations $\mathcal{A}$ has
the following: for any smooth function $\Phi$ with $r$ arguments and a collection $I_1,\dots,I_r\in\mathcal{A}$
we get $\Phi(I_1,\dots,I_r)\in\mathcal{A}$ (whenever the composition is defined).

We will not consider arbitrary smooth invariant functions, but will restrict to rational-polynomial
invariant functions, which are defined globally and form a subalgebra
$\mathfrak{P}_l(\E)^G\subset C^\infty(\E^\infty)$ (as in the Introduction).
In this case the above composition $\Phi$ must be taken rational-polynomial as well.
However the above algebraic operations are not usually enough to finitely generate $\mathcal{A}$ or $\mathfrak{P}_l(\E)^G$.

Sophus Lie proposed to produce higher order invariants via invariant derivations $\nabla$
of the algebra of differential invariants (\cite[\S4.4]{Li$_2$}). He suggested that a finite number of them $\nabla_1,\dots,\nabla_n$ is enough to generate the entire algebra $\mathcal{A}$.

Consider the bundle $\t_k=\pi_{k,1}^*(\t)$ on the jet-space $J^k(M,n)$, which is the
pull-back of the tautological bundle $\t$ over $J^1(M,n)$. Its fiber at the point $a_k$
can be identified with the horizontal space $L(a_k)$ defined in \S\ref{S11}.
Let $v\mapsto\hat v$ denote this natural lift $\t(a_1)\to\t_k(a_k)$.

Denote by $C^\infty(\t_k)$ the $C^\infty(J^k)$-module of sections of $\t_k$,
and let $C^\infty(\t_k|\E^k)$ be its restriction to the equation
in order $k$. Its elements are finite sums $\sum f_iv_i$, where $f_i\in C^\infty(\E^k)$
and $v_i$ are sections of the bundle $\t_k$.

Pullback $\pi^*_{k+1,k}$ injectively maps these modules and simultaneously extends their rings.
The $C^\infty(\E^\infty)$-module of {\em horizontal vector fields\/}
(also: $\mathcal{C}$-fields \cite{KLV}) is the inductive limit
 $$
C^\infty(\E^\infty,\t)=\lim_{\rightarrow}C^\infty(\t_k|\E^k).
 $$
The $C^\infty(\E^\infty)$-module of horizontal 1-forms $C^\infty(\E^\infty,\t^*)$
is defined similarly.

The horizontal differential $\hat df$ of a function $f\in C^\infty(\E^k)$ at $a_{k+1}\in\E^{k+1}$
is given by $\hat d f=d_{a_k}f|_{L(a_{k+1})}$.
This gives the map
 $$
\hat d:C^\infty(\E^\infty)\to C^\infty(\E^\infty,\t^*).
 $$

The action of horizontal vector fields on functions comes from the natural pairing between $\t$ and $\t^*$:
 $$
\xi(f)=i_\xi\hat df=i_{\hat\xi}df,\qquad \xi\in C^\infty(\E^\infty,\t),\ f\in C^\infty(\E^\infty).
 $$
This action is linear and satisfies the Leibniz rule. Therefore any horizontal vector field $\xi$
is a derivation of the algebra of functions $C^\infty(\E^\infty)$.

In addition such derivations, when naturally extended to derivations of the exterior algebra $\Omega^\bullet(\E^\infty)$ (by the condition of commutation with de Rham $d$),
preserve the Cartan ideal $\mathcal{C}\Omega\subset\Omega^\bullet(\E^\infty)$,
generated by 1-forms $U(f)=df-\hat d f$, $f\in C^\infty(\E^\infty)$ \cite{KLV}.
For the jets of sections $J^\infty\pi$ (and equations in them) the $\mathcal{C}$-fields
are combinations of the total derivatives of the first order.


 \begin{dfn}
Invariant derivatives are $G$-invariant elements of the module $C^\infty(\E^\infty,\t)$
(according to the natural action of the pseudogroup $G$ on it).
We denote the space of invariant derivatives by $\mathfrak{D}(\E,\t)$.
 \end{dfn}

The space $\mathfrak{D}(\E,\t)$ has the natural structure of a module over the algebra of
invariant functions. We have two possibilities for this module: (micro-local) smooth
or global rational-polynomial depending on the choice of the algebra $\mathcal{A}$.
The space $\mathfrak{D}(\E,\t)$ is also a Lie algebra acting on $\A$ from the left.
So we get

 \begin{prop}\label{jkl}
Every $\nabla\in\mathfrak{D}(\E,\t)$ determines a $G$-invariant derivation
 $$
\nabla:\mathcal{A}\to\mathcal{A}.
 $$
 \end{prop}

As noted above these derivations of the algebra $\A$ are special: they preserve the Cartan
ideal $\mathcal{C}\Omega$ of the equation; equivalently they preserve the Cartan distribution
$\mathcal{C}\subset T\E^\infty$ generated by the $\mathcal{C}$-fields.

 \begin{dfn}\label{DfnDer}
Consider the Lie algebra of $\mathcal{C}$-fields on $\E^\infty$ as
derivations of the algebra of functions $C^\infty(\E^\infty)$.
Let $\mathfrak{N}$ be its Lie subalgebra leaving invariant the subalgebra of differential
invariants $\mathcal{A}\subset C^\infty(\E^\infty)$ and $\mathfrak{N}_0\subset\mathfrak{N}$ be
the ideal of the fields that act trivially on $\mathcal{A}$.
Define the invariant derivations as the quotient
 $$
\mathfrak{Der}(\E,\t)=\mathfrak{N}/\mathfrak{N}_0.
 $$
 \end{dfn}

For completeness we should have included $G$ into the notation for both
$\mathfrak{D}(\E,\t)$ and $\mathfrak{Der}(\E,\t)$, but we skip it for simplicity sake.
As with the former, the module of invariant derivations $\mathfrak{Der}(\E,\t)$
over $\mathcal{A}$ can be smooth of rational-polynomial depending on the choice
of the algebra $\mathcal{A}$.



Proposition \ref{jkl} yields the homomorphism
(of both Lie algebra and $\A$-modules structures)
 $$
\iota:\mathfrak{D}(\E,\t)\to\mathfrak{Der}(\E,\t),
 $$
which in general is neither injective nor surjective.
We will investigate these two spaces and relations between them in Sections \ref{S23}, \ref{S33}.

\subsection{Particular case: Tresse derivatives}\label{S14}

Let us first write the constructions from the previous section in local coordinates.
We identify $M$ with the total space of a fiber bundle $\pi:M\to B$ with $\dim B=n$.
Then $J^k\pi\subset J^k(M,n)$ is a local chart near $a_k=[N]_a^k$ if the fibers of $\pi$
are transversal to the germ of the submanifold $N\subset M$ at $a$.

Choose local coordinates $(x^i,u^j)$ on $M$ adapted to $\pi$, i.e.\ $x^i$ are coordinates
on the base and $u^j$ are coordinates on the fibers.
These local coordinates on $\pi$ induce the canonical coordinates $(x^i,u^j_\z)_{0\le|\z|\le k}$ on
$J^k\pi$, where for a submanifold $N$ given as graph $u=h(x)$ we have:
$u^j_\z([h]_a^k)=\dfrac{\p^{|\z|}h^j}{\p x^\z}(a)$.

In the local chart $J^\infty(\pi)\subset J^\infty(M,n)$ the derivatives (horizontal vector fields)
can be written as $\sum f_i\,\D_i$, where
 $$
\D_i=\p_{x^i}+\sum_{j;\z}u^j_{\z+1_i}\p_{u_\z^j}:C^\infty(J^k\pi)\to C^\infty(J^{k+1}\pi)
 $$
are the total derivative operators.

The horizontal differential 
has the following form in these coordinates:
 $$
\hat df=\sum\D_i(f)\,dx^i.
 $$
The same formulae apply for the restriction to the equation $\E$.

An important particular case of derivations of the algebra $C^\infty(\E^\infty)$
constitute derivatives \`a la Tresse, which we now introduce.

Consider $n$ functions $f_1,\dots,f_n$ on $\E^k$ such that the
open\footnote{When we turn to the algebraic situation, then 'open' becomes 'Zariski open'.} set
 \begin{equation}\label{ndgTr}
\E_{k+1}'=\{a_{k+1}\in\E^{k+1}\,:\,df_1\we\dots\we df_n|_{L(a_{k+1})}\ne0\}
 \end{equation}
is dense in $\E^{k+1}$. Then we define differential operators
 $$
\hat\p_i:C^\infty(\E^k)\to C^\infty(\E_{k+1}'),
 $$
by the formula
 $$
df|_{L(a_{k+1})}=\sum_{i=1}^n\hat\p_i(f)(a_{k+1})\,df_i|_{L(a_{k+1})},
 $$
The expressions $\hat\p_i(f)=\hat\p f/\hat\p f_i$ are called the Tresse derivatives of $f$ with respect to $f_i$.
These derivations enjoy the commutativity property $[\hat\p_i,\hat\p_j]=0$.

 \begin{prop}
If $f_1,\dots,f_n$ are $G$-differential invariants, then the operators
$\hat\p_i=\hat\p/\hat\p f_i:\mathcal{A}_k\to\mathcal{A}_{k+1}$ are invariant derivations.
 \end{prop}

On $G$-invariant set $\E_{k+1}'$ we can re-write the condition in (\ref{ndgTr}) as
  $$
\hat df_1\we\dots\we\hat df_n\ne0,
 $$
i.e.\ the Jacobian $DF=\|\D_i(f_j)\|$ is non-degenerate. For any other $f\in\mathcal{A}$ we have:
 \begin{equation*}
\hat df=\sum_i\hat\p_i(f)\,\hat df_i.
 \end{equation*}

Thus
 $$
\hat d=\sum dx^i\ot\D_{x^i}=\sum\hat df_i\ot\hat\p/\hat\p f_i,
 $$
which yields the expression of the Tresse derivatives as $\C$-fields:
 \begin{equation*}
\hat\p_i=\sum_j\bigl(DF^{-1}\bigr)_{ij}\D_{x^j},
 \end{equation*}
This formula can be interpreted as the chain rule in total derivatives.

\subsection{Characteristics and involutivity}\label{S15}

Let $g=\oplus g_k$ be a symbolic system (it can correspond to a differential equation $\E$, but in this section
we discuss pure algebraic results) that is a collection of subspaces $g_k\subset S^k\t^*\ot\nu$ subject only to one
constraint $g_k\subset g_{k-1}^{(1)}$ (in terms of the Spencer cohomology this writes as $H^{k,0}(g)=0$).
For a vector space $V$ let us denote by $SV=\oplus S^iV$ the ring of homogeneous polynomials on $V^*$.

For $v\in\t$ define the map $\delta_v:S^{k+1}\t^*\ot\nu\to S^k\t^*\ot\nu$ by the formula
$\delta_v(p)=\langle v,\delta p\rangle$. More generally given $v_{i_j}\in V$ and
$v=\sum v_{i_1}\cdots v_{i_r}\in S^rV$ define
$\d_v=\sum\d_{v_{i_1}}\!\cdots\,\d_{v_{i_r}}:S^{k+r}\t^*\ot\nu\to S^k\t^*\ot\nu$.

The $\R$-dual (or $\CC$-dual if we work over the field of complex numbers) system
$g^*=\oplus g_k^*$ is an $S\t$-module with the structure given by
 $$
(v\cdot\kappa)p=\kappa(\d_vp),\ v\in S\t,\ \kappa\in g^*,\ p\in g.
 $$
This module, called the {\em symbolic module\/}, is Noetherian and
the Spencer cohomology of $g$ dualizes to the Koszul homology of $g^*$ \cite{GS$_1$}.

The {\em characteristic ideal\/} is defined by $I(g)=\op{Ann}(g^*)\subset S\t$.
Notice that passing to the module $\mathcal{M}=S\t/I(g)$ results in the shift of indices
in the homology: $H^{i,j}(g)^*=H_{i,j}(g^*)=H_{i+1,j-1}(\mathcal{M})$ --
the Koszul homology of $\mathcal{M}$.

A symbolic system has maximal order $l$ if $H^{k,1}(g)=0$ for $k>l$ (the cohomology $H^{*,1}$ counts the orders).
Referring to \cite{KL$_3$} for the general definition in the case of systems of different orders, we adapt the
following version here: a symbolic system $g$ is involutive from the level $k_0$ if $H^{k,j}(g)=0$ for $k\ge k_0$
and all $j$. By the Poincar\'e $\delta$-lemma \cite{S$_2$} any symbolic system of fixed orders
is involutive starting from some $k_0$.
The bound $k_0$ can be taken universal, depending only on $n=\dim\t$ and the orders of $g$.

The {\em affine real characteristic variety\/} of $g$ (or of $\E$) is the set of
$p\in\t^*\setminus\{0\}$ such that for every $k$ there exists a $w\in\nu\setminus\{0\}$ with $p^k\ot w\in g_k$.
This is a conical affine variety. The {\em affine complex characteristic variety\/} is defined similarly:
the set of $p\in\t^*_\CC$ such that $\exists\,w\in\nu_\CC$ with $p^k\ot w\in g_k^\CC\setminus\{0\}$.
Its projectivization is the (complex) {\em characteristic variety\/} $\op{Char}^\CC(g)\subset \mathbb{P}^\CC\t^*$
(the set of all $w$ for characteristic $[p]=\CC\,p$ form the {\it kernel sheaf\/} over the characteristic variety).

Relation of the characteristic variety to the characteristic ideal $I(g)=\oplus I_k$ is
given by the formula:
 $$
\op{Char}^\CC(g)=\{p\in \mathbb{P}^\CC\t^*\,|\,f(p^k)=0\,\forall f\in I_k,
\forall k\}.
 $$

Note that dimension $d$ of the affine complex characteristic variety
equals to the Chevalley dimension of the symbolic module $g^*$. Recall also that a
sequence of elements $v_1,\dots,v_s\in\t$ is called regular if
$v_i$ is not a zero divisor in the $S\t$-module $g^*/(v_1,\dots,v_{i-1})g^*$ for all $i\le s$.

An element $v\in\t$ is regular for sufficiently prolonged system $g$ iff the
hyperplane $\mathbb{P}\op{Ann}(v)^\CC$ does not contain
the characteristic variety $\op{Char}^\CC(g)$. More generally, a sequence $(v_1,\dots,v_i)$
is regular iff $\op{Char}^\CC(g)$ meets $\mathbb{P}\op{Ann}(v_1,\dots,v_i)^\CC$ transversally.
In other words, the projection of the characteristic variety along annihilator to
$\mathbb{P}(\langle v_1,\dots,v_i\rangle^\CC)^*$ is surjective for $i<d$ and injective elsewise.

It follows that there exists a sequence $(v_1,\dots,v_n)$ in $\t$ that is regular
for the module $g^*_{[k_0]}=\oplus_{i\ge k_0}g_i^*$. This implies
(see \cite{AB} or the appendix with a letter of Serre in \cite{GS$_1$}) that all Koszul homology
of $g^*$ w.r.t. the sequence $(v_1,\dots,v_n)$, or equivalently with coefficients in $\t$,
vanish except perhaps for the zero grading (equivalently $g^*_{[k_0]}$ is a Cohen-Macaulay module over $S\t$).
Dualizing this statement we obtain again that the positive Spencer cohomology $\oplus_{k\ge k_0}H^{k,+}(g)$ vanishes.
When $g\subset S\t^*\ot\nu$ the zero cohomology vanishes as well.

The above theory is applicable to the pseudogroup $G$, which is a particular
case of an equation (nonlinear Lie equation). The corresponding symbolic group $\g$
induces the characteristic variety $\op{Char}^\CC(\g)\subset \mathbb{P}^\CC T^*$, $T=T_aM$,
which (contrary to the case of a general equation $\E$) has the same projective type
at different points $\vp_k\in G^k_a$.

\section{The structure of the space of $G$-orbits}\label{S2}

In this section we study the orbits of the pseudogroup $G$ action on $\E^k$
as the jet-level $k$ grows and we prove that the local behavior of differential invariants is
governed by a certain symbolic system.
We keep throughout our general assumptions that the action $G$ is algebraic on $\E$ and
transitive on the base $M$, though in some cases these assumptions can be relaxed.

Together with the pseudogroup $G$ we will consider its Lie algebra sheaf $\mathcal{G}$
defined as follows. Consider an isotopy $\vp_t\in G$, $t\in(-\ve,\ve)$, $\vp_0=\op{id}_M$,
i.e.\ a path in the pseudogroup through the unity smoothly depending on the parameter $t$,
and let $X(a)=\left.\frac{d}{dt}\right|_{t=0}\vp_t(a)$ be the vector field defined
in a neighborhood $U\subset M$ contained in the domain of $\vp_t$.
It is clear that the collection of all such $X$ form a subsheaf $\mathcal{G}$ of the sheaf
of all germs of vector fields on $M$, and moreover that $X\in\mathcal{G}$ satisfy the
linearized equation of the pseudogroup at the unity.

It is not clear in general that $\mathcal{G}$ gives all smooth solutions of this linear equation
(cf. \cite{SS}), but we do not rely upon this property, as we only need the jets of the fields $X$
from $\mathcal{G}$ to describe the differential invariants of $G$.
Indeed, the differential invariants $I\in\A$ satisfy the Lie equation $X^{(k)}\cdot I=0$,
where $X^{(k)}$ is the prolongation of $X\in\mathcal{G}$ to the space of $k$-jets,
and they are determined by it if the pseudogroup $G$ is connected.

Starting from this section we write $\A$ instead of $\mathfrak{P}_l(\E)^G\subset\mathfrak{R}(\E)^G$
(with $l$ being for some time undetermined) unless otherwise specified.

\subsection{Quotient by the algebraic groups on a finite jet level}\label{S21}

Let us begin with some basic facts about actions of an algebraic Lie group $G$ on an algebraic variety $X$;
see \cite{PV,SR} for details.
Such $X$ comes naturally equipped with the sheaf of functions, whose section algebra
$\mathcal{O}_X(U)$ over a (Zariski) open subset $U\subset X$ consists of the restrictions of the
regular functions.

A surjective open morphism onto another algebraic variety $p:X\to Y$ is called a
{\it geometric quotient\/} if the fibers $p^{-1}(y)$ are $G$-orbits
and for every open subset $U\subset Y$ the map
$\mathcal{O}_Y(U)\to\mathcal{O}_X(p^{-1}(U))^G$ is an isomorphism of algebras.

A geometric quotient is necessary a categorical quotient, i.e.\ any morphism
$\rho:X\to Z$ constant along the orbits of $G$ can be factorized through $p$, namely
$\rho=f\circ p$ for a morphism $f:Y\to Z$. Thus in categorical sense if the geometric quotient
$Y=X/G$ exists, then it is unique.

However the geometric quotient does not always exist due to the presence of singular orbits
(regular orbits have the maximal possible dimension and they fiber a neighborhood
after possibly restricting to an open dense subset).
A way out of this difficulty is to remove the singular orbits.

 \begin{theorem}{{\bf[Rosenlicht]}}
For an algebraic action of the Lie group $G$ on an irreducible variety $X$
a finite set of rational invariants separates general orbits.
 \end{theorem}

Let us remark that while in the original paper \cite{R$_1$} the field of definition is allowed
to be $\R$, in other sources \cite{PV,SR} the standing assumption is that the field is
algebraically closed. However the statement over $\CC$ implies the one over $\R$ via the
complexification argument as follows.
If the action is real, the invariants for the complexification can be chosen
real rational functions (the real and imaginary parts of a rational invariant are $G$-invariant as well).
If they form a coordinate system near a regular orbit in the complexification, they
will also form a coordinate system when restricted to the totally real part
(a Zariski open set of which consists of regular orbits by the analytic continuation argument).

 \begin{cor}
Under the above assumptions, on a nonempty open set $X_0\subset X$
the action of $G$ admits a geometric quotient $p:X_0\to Y_0$.
 \end{cor}

To interpret (and explain) geometrically the above corollary, let us first treat the situation over $\CC$
(we refer to \cite{E,La} for details on the involved notions).
Consider the field of rational invariants $\mathfrak{R}(X)^G$ and let its elements $y_1,\dots,y_d$
separate the general $G$-orbits. Denote by $R[y]$ the polynomial ring generated by them. Let $Y=\op{Spec}R[y]$
be its spectrum (i.e.\ the set of all prime ideals in $R[y]$ equipped with the Zariski topology),
and let $Y_m=\op{maxSpec}R[y]$ be the maximal spectrum (i.e.\ the subset of all maximal ideals).
If $I$ is the polynomial ideal generated by the relations among $\{y_i\}_{i=1}^d$ in $\CC[y_1,\dots,y_d]$,
the space $Y_m$ can be identified with the locus $Z(I)=\{z\in\CC^d(y_1,\dots,y_d):f(z)=0\,\forall f\in I\}$
with the coordinate ring $R[y]=\CC[y_1,\dots,y_d]/I$. Since $I$ is obviously radical, it corresponds bijectively
to $Y_m=Z(I)$ by Hilbert's Nullstellensatz.

Embedding $R[y]\subset\mathfrak{R}(X)^G\subset\mathfrak{R}(X)$
induces a dominant rational mapping $\bar{p}:X\to Y_m$, $\bar{p}(x)=(y_1(x),\dots,y_d(x))$
(dominant means that the image is Zariski dense; indeed, if there is a polynomial relation on the image, it
shall be among the relations on $\{y_i\}_{i=1}^d$ defining $I$).
Since $\op{maxSpec}R[y]$ is Zariski closed in $\op{Spec}R[y]$, we can also treat $\bar{p}:X\to Y$
as a dominant rational mapping.

This morphism $\bar{p}$ is the rational quotient of $X$ by $G$,
having the universal property similar to the described above categorical quotient.
The base $Y_0$ of the geometric quotient can be embedded into $Y$ as a nonempty open subset.
Codimension of the general orbit of $G$ on $X$ equals $d=\dim Y_m$ and
can be also described as the {\it transcendence degree\/} of the field $\mathfrak{R}(X)^G$.

Consider now the situation over $\R$. Here the maximal spectrum does not reflect the geometry of $Y_m$,
so we consider the variety $Z(I)\subset\R^d$ instead (it is a Zariski closed subset in $\op{maxSpec}R[y]$).
Notice that $I$ is a (finitely generated) polynomial ring with real coefficients (as $X$ and $y_i$ are real)
that is related to its complex counterpart $I_\CC$ as follows: $I=I_\CC\cap\R[y_1,\dots,y_d]$.
 \begin{lem}
The ideal $I\subset\R[y_1,\dots,y_d]$ is real radical and prime.
 \end{lem}

Recall \cite{D} that the real radical of an ideal $I$ in a (commutative) ring $\mathcal{R}$ is
 $$
\sqrt[R]{I}=\{f\in\mathcal{R}\,|\,\exists m,n\in\mathbb{N},\,k_i\in\R_+,\,u_i\in\mathcal{R}\,(1\le i\le n):
(1+\sum_{i=1}^n k_iu_i^2)f^m\in I\}
 $$
and $I$ is real radical whenever $\sqrt[R]{I}=I$.
\smallskip

 \begin{proof}
Clearly, if $f=f(y_1,\dots,y_d)\in\R[y_1,\dots,y_d]$ satisfies the above relation, then $f(x)=0$ $\forall x\in X$
whence $f\in I$, where we shorten $f(x)=f(y_1(x),\dots,y_d(x))$. Thus $I$ is real radical.

If for $f,g\in\R[y_1,\dots,y_d]$ we have $f\cdot g\in I$, then $f(x)\cdot g(x)=0$ $\forall x\in X$ and
(real) irreducibility of $X$ implies that either $f(x)=0$ $\forall x\in X$ or $g(x)=0$ $\forall x\in X$,
i.e.\ $f\in I$ or $g\in I$. This $I$ is prime.
 \end{proof}

Now the ideal of $Y=Z(I)\subset\R^d(y_1,\dots,y_d)$ is $I$ by the real version of Nullstellensatz \cite{D},
and $Y$ is real irreducible with the coordinate ring $R[y]=\R^d[y_1,\dots,y_d]/I$. Since $I$ is prime,
$R[y]$ is an integral domain and its field of quotients is easily identified with $\mathfrak{R}(X)^G$.
Thus we obtain the rational quotient $\bar{p}:X\to Y$ in the real case too
(if $Y_m^\CC$ is the rational quotient of the complexified action, then $Y_m=Y_m^\CC\cap\R^d$ is its real slice,
so the above arguments show the variety $Y^\CC$ is rationally equivalent to the complexification of the real variety $Y$).


\medskip

Let us now apply this general theory to our (real) case.
Since the action of $G$ on $\E$ is transitive on the base, all orbits project onto $M$. So to study
the space of orbits on the level of $k$-jets it is enough
to restrict to one fiber $\E^k_a$ and the action of $G^k_a$ on it. We have:

 \begin{lem}
The algebraic manifold $\E^k_a$ is irreducible for all $k$ (and $a\in M$).
 \end{lem}

 \begin{proof}
Let us use the induction on $k$. The base of induction $k=l$ holds by our hypothesis.
Assume the claim on the jet-level $(k-1)$. Since (also by hypothesis) $\E$ is formally integrable,
$\E^k_a$ is an affine bundle over $\E^{k-1}_a$. Let $f,g$ be two polynomials in $k$-jets (with fixed point $a$),
and assume $(f\cdot g)(a_k)=0$ $\forall a_k\in\E^k_a$, and also $f\not\equiv0$.

For any point $a_k$ such that $f(a_k)\neq0$ consider the fiber through it:
$F(a_{k-1})=\pi_{k,k-1}^{-1}(a_{k-1})$, $a_{k-1}=\pi_{k,k-1}(a_k)$. Since it is affine (and hence irreducible),
$(f\cdot g)|_{F(a_{k-1})}=0$ implies $g|_{F(a_{k-1})}=0$. But $f(a_k)\neq0$ for Zariski generic $a_k\in\E^k_a$, so
$g=0$ on all fibers over a Zariski open set in $\E^{k-1}_a$. This by irreducibility of $\E^{k-1}_a$ implies
that $g$ is identically zero.
 \end{proof}

Our original assumption is that the real algebraic variety $\E^k_a$ is
irreducible, but this implies that its complexification is irreducible as
a complex algebraic variety. Indeed, consider the real ideal $I(\E_a^k)$
in the algebra of (inhomogeneous) polynomials on $ J^k_a$, corresponding to $\E_a^k$.
The corresponding complex ideal $I(\E_a^k)\otimes\CC$ in the algebra of
complex polynomials ${\mathcal{P}}(\CC^N)^*$, $N=\dim J^k_a$, has the zero-locus
${}^\CC\E_a^k\subset{}^\CC\!J^k_a$, which is the complexification of $\E_a^k$.
This complex variety is irreducible if $I(\E_a^k)\otimes\CC$ is prime, i.e.\
the algebra ${\mathcal{P}}(\CC^N)^*/(I(\E_a^k)\otimes\CC)$ is an integral domain.
Assume that $F\cdot G=0\mod(I(\E_a^k)\otimes\CC)$
for two complex polynomials $F,G$. Multiplying this by $\bar F\bar G$ and restricting to the real
variety $\E_a^k$ we get $|F|^2\cdot|G|^2=0$ along $\E_a^k$. By real irreducibility this yields
the required alternative $F=0$ or $G=0$.

Thus by the Rosenlicht's theorem, the algebraic action has a geometric quotient outside
a $G$-stable Zariski closed (and so nowhere dense) subset.

 \begin{prop}\label{RSL}
For every $k$ there exists a proper closed set $S_k\subset \E^k$ such that the complement admits
the geometric quotient $\bigl(\E^k\setminus S_k\bigr)/G^k\simeq Y^k_0$.
 \end{prop}

Notice that $S_k$ is Zariski closed in the sense that $S_k\cap\E^k_a$ is Zariski closed for every
$a\in M$. Then we identify $\bigl(\E^k\setminus S_k\bigr)/G^k$ with
$\bigl(\E^k_a\setminus(S_k\cap\E^k_a)\bigr)/G^k_a$ and the above algebraic machinery applies.

\subsection{A symbolic system associated to $G$-orbits}\label{S22}

The orbit of the $G$-action through a $k$-jet $a_k\in\E_k$ is $G^k\cdot a_k\subset\E^k$.
Let $\Delta_k(a_k)=T_{a_k}(G^k\cdot a_k)$ be the tangent differential system. This is easily seen
to coincide with the distribution of evaluations of $k$-jets of the Lie algebra sheaf $\mathcal{G}$
of the pseudogroup $G$ at the point $a_k$: $\Delta_k(a_k)=\{X^{(k)}_{a_k}:X\in\mathcal{G}\}$.

The distribution $\Delta_k|_{\E_k'}\subset T\E_k'$ is obviously integrable on the
set of regular orbits $\E_k'\subset\E^k$. Differential invariants of
order $k$ are its first integrals. Notice that $\pi_{k,k-1}:\E_k'\to\E_{k-1}'$.
We call a point $a_k$ {\em regular\/} if $a_k\in\E_k'$.
By Rosenlicht's theorem \cite{R$_1$,SR} the set of regular points $\E_k'$ is Zariski open,
and the above first integrals separate the regular orbits of $G^k$
(recall that the action of $G^k$ on $\E^k$ is equivalent to the vertical action of $G^k_a$ on $\E^k_a$).

For a point $a_\infty\in\E^\infty$ consider a collection of subspaces
 $$
\varpi_k=\op{Ker}\bigl(d\pi_{k,k-1}:\Delta_k(a_k)\to\Delta_{k-1}(a_{k-1})\bigr)\subset S^k\t^*_a\ot\nu_a.
 $$

 \begin{prop}\label{prp4}
There exists an integer $l$ and a $G$-invariant Zariski open subset $\E_l''\subset\E^l$ such that
$\pi_{\infty,l}(a_\infty)=a_l\in\E_l''$ implies $\varpi_{k+1}\subset\varpi_k^{(1)}$ for all $k\ge l$.
 \end{prop}

When the point $a_\infty$ is fixed, this corresponds to Lemmata 22.5 and 23.1 of \cite{Kum}
(see also the idea of B.\,Malgrange on p.357; at this point nothing more than the usual
regularity is assumed in loc.cit.).
It is possible to show that with our algebraic assumptions, his proof extends to all points $a_\infty$
in a Zariski open subset, but we prove it independently for completeness.

\smallskip

 \begin{proof}
Denote by $\mathcal{G}^k_{a_1}$ the algebra of $k$-jets of vector fields from $\mathcal{G}$ vanishing at $a_1$
(this is a finite-dimensional Lie algebra). For any point $a\in M$ consider the bundle $\op{St}^k_{a}$ over
$\E^k_a$, whose fiber over $a_k$ is $\op{st}(a_k)=\{X^{(k)}\in\mathcal{G}^k_{a_1}:X^{(k)}_{a_k}=0\}$.
Consider the bundle $V^\tau_a$ over $\E^1_a\subset\op{Gr}_n(T_aM)$ whose fiber over $a_1$
is $\op{End}(\tau_a)$ (recall that we identify $a_1$ with $\tau_a$). Define the map
$\psi_k:\op{St}^k_a\to V^\tau_a$ by the formula $X\mapsto d_aX$. Both spaces $\op{St}^k_a$ and $V^\tau_a$
are algebraic and so is the map $\psi_k$ (clearly $\op{Im}(\psi_k)\cap\op{End}(\tau_a)$
is a vector subspace in $\op{End}(\tau_a)\subset V^\tau_a$ $\forall a_1\in\E^1_a$).

By Chevalley's theorem on constructible sets \cite[Corollary 14.7]{E} the Zariski closure 
$\Psi^k_a$ of $\op{Im}(\psi_k)$ has this image as an open dense subset (in the usual topology).
Since the sequence of algebraic varieties $\Psi^k_a$ is nested (decreasing
because $X^{(k+1)}_{a_{k+1}}=0$ $\Rightarrow$ $X^{(k)}_{a_k}=0$) it stabilizes: for some $l$ we have
$\Psi^k_a=\Psi^l_a$ for $k\ge l$. We can choose a finite number of rational sections of $\op{St}^l_a$
whose $\psi_l$-images span $\Psi^l_a$.
The open set over which the sections are defined is $(\E_l'')_a$, and uniting these over all
$a\in M$ we obtain the set $\E_l''\subset\E^l$. For every $a_k\in\E^k$ with $\pi_{k,l}(a_k)=a_l\in\E''_l$
we get: $\psi_k(\op{st}(a_k))$ is the stable (independent of $k$) subspace $\psi_l(\op{st}(a_l))\subset\op{End}(\tau_a)$.
In other words, for any $\pi_{k+1,k}(a_{k+1})=a_k$ with $\pi_{k,l}(a_k)\in\E_l''$ and $X\in\op{st}(a_k)$
there exists $\tilde X\in\op{st}(a_{k+1})$ such that $\psi_{k+1}(\tilde X^{(k+1)})=\psi_k(X^{(k)})$,
i.e.\ $d_a\tilde X=d_aX$.

The proposition claims that the sequence $0\to\varpi_k\stackrel{\d}\to\varpi_{k-1}\ot\t^*$ is exact.
By injectivity of the Spencer $\d$-differential at the first term, this is equivalent to
$\delta_\xi:\varpi_k\to\varpi_{k-1}$ $\forall\xi\in\t$.
Consider $X\in\mathcal{G}$ with the lift $q_k=X^{(k)}_{a_k}\in\Delta(a_k)$
to the point $a_k\in\E^k$ such that $X^{(k-1)}_{a_{k-1}}=0$, i.e.\ $q_k\in\varpi_k$
(in particular, $X_a=0$). We can modify $X$ by the above $\tilde X$ to achieve $d_aX=0$.
For $Y\in\mathcal{G}$ with $\xi=Y(a)\in\t_a$
the Lie bracket $L_YX=[Y,X]$ lifts to $q_{k-1}=[Y,X]^{(k-1)}_{a_{k-1}}$ at $a_{k-1}\in\E^{k-1}$.
Moreover $q_{k-1}$ depends only on the values $\xi$ and $q_k$, and $[Y,X]^{(k-2)}_{a_{k-2}}=0$.
This means that $q_{k-1}\in\varpi_{k-1}$, and provided $a_k$ is regular, we conclude: $\delta_\xi(q_k)=q_{k-1}$.

Let us explain the above abstract argument in local coordinates $(x^i,u^j)$ on $M$, when the submanifolds
write as $N=\{u=h(x)\}$. The vector fields from $\mathcal{G}$ have the form $X=\a^i(x,u)\p_{x^i}+\b^j(x,u)\p_{u^j}$
and their prolongations are
 $$
X^{(k)}=\a^i\D_i^{(k+1)}+\sum_{|\z|\le k}\D_\z(\vp_X^j)\p_{u^j_{\z}},
 $$
where $\D_i^{(k+1)}=\p_{x^i}+\sum_{|\z|\le k}u^j_{\z i}\p_{u^j_{\z}}$ is the truncated total derivative,
$\D_\z=\D_{i_1}\cdots\D_{i_s}$ for a multi-index $\z=(i_1,\dots,i_s)$ is the iterated total derivative and
$\vp_X^j=\b^j-\a^i u^j_i$ is the generating function\footnote{In several sources, e.g.\ \cite{Ol$_1$},
$\vp$ is called characteristic, but in our case this would lead to a confusion with characteristic covectors,
and we use the terminology adapted in \cite{KLV}.} of (the lift of) $X$.
In other words, the generating function (section) $\vp_X\in C^\infty_{\text{loc}}(J^1(M,n),\nu)$ is given
by the formula $\vp_X(a_1)=X_a\,\op{mod}\t_a$,
where $a_1=[N]_a^1$, $\t_a=T_aN$.

Denote restriction of $\vp_X$ to the submanifold $N$ by
$\vp^j_{X|h}=\b^j_{|h}-\a^i_{|h}h^j_i(x)$, where
$\a^i_{|h}(x)=\a^i(x,h(x))$, $\b^j_{|h}(x)=\b^j(x,h(x))$.
The condition $X\in\op{st}(a_k)$ is equivalent to $\a^i(a)=0$, $\vp^j_{X|h}\in\mu^k_a$; here $\mu_a$ is the
(maximal) ideal of functions vanishing at $a\in M$. If $k\ge l$, then $X$ modulo $\op{st}(a_{k+1})$ can be presented as
a vector field from $\mathcal{G}$ with $\a^i(a)=0$ and $d_a\a^i|_{\tau_a}=0$. In this case
 \begin{multline*}
[Y,X]_{a_{k-1}}^{(k-1)}=[Y^{(k-1)},X^{(k-1)}]_{a_{k-1}}=L_Y(\a^i_{|h})_a\D_i^{(k)}\\
+\sum_{|\z|\leq k-1}(L_Y\D_\z(\vp^j_{X|h}))_{a_k}\p_{u^j_{\z}}
=\sum_{|\z|=k-1}(\D_\z(L_Y\vp_{X|h}^j))_{a_k}\p_{u^j_{\z}}
 \end{multline*}
If $X$ corresponds to $q_k\in\varpi_k$ and $Y$ to $\xi\in\tau_a$, then the last expression corresponds to
$q_{k-1}=\delta_\xi(q_k)\in S^{k-1}\tau_a^*\ot\nu_a$, and it belongs to $\varpi_{k-1}$.
 \end{proof}

 \begin{remark}
Thus $\varpi(a_\infty)_{\ge l}=\{\varpi_k(a_\infty)\}_{k=l}^\infty$ is a symbolic system from the level $l$.
Define $\tilde\varpi_k=\varpi_k$ for $k\ge l$ and $\tilde\varpi_k=\{\delta_\xi\tilde\varpi_{k+1}:\xi\in\tau\}$
successively for $k=l-1,\dots,0$. The new system $\tilde\varpi=\{\tilde\varpi_k\}_{k\ge0}$ is symbolic
(from the level 0); we call it the completed symbolic system.
 \end{remark}

The next statement can have a different meaning of $l$ (and $\E''_l$) than that used
in Proposition \ref{prp4} (the one from above will be denoted $l_0$ below).

 \begin{theorem}\label{trma}
There exists an integer $l$ and a $G$-invariant Zariski open subset $\E_l''\subset\E^l$ such that
$\pi_{\infty,l}(a_\infty)=a_l\in\E_l''$ implies $\varpi_{k+1}=\varpi_k^{(1)}$ for all $k\ge l$.
 \end{theorem}

The proof below has the following idea. We consider the algebra $\mathfrak{A}$ of polynomials
on the bundle $\t^*$ over $\E^\infty$ with coefficients in the field of rational functions on $\E^\infty$
(i.e.\ the algebra of sections of $S\tau$).
Then we define the space $\mathfrak{M}$ of sections of the sheaf $\varpi^*_{[l_0]}$ over $\E^\infty$ with
the same rational coefficients and prove it is a Noetherian module over $\mathfrak{A}$
(this is done via evaluations at generic points since $\varpi^*_{[l_0]}$ is a Noetherian module over $S\t$).
Then the generating set gives the required $G$-invariant neighborhood in $\E^\infty$
and the symbolic property for $\varpi^*$.

\smallskip

 \begin{proof}
Fix a point $a\in M$ (its choice is not important because the pseudogroup $G$ acts transitively on $M$).
We will work with algebraic manifolds $\E_a^k$, which we can complexify. Since the statement
over $\CC$ clearly implies that one over $\R$, we assume till the end of the proof that all our systems are
complex.

Consider the tautological bundle $\t$ over $\E^1_a\subset\op{Gr}_n(T_aM)$ and denote the space of its 
rational 
sections by $\Gamma(\tau)$. Similarly, $\Gamma(S\tau)$ is the algebra of sections of the associated
symmetric bundle, and it can be identified with the graded commutative algebra of
functions on the bundle $\tau^*$, that are polynomial on the fibers.
Extending the coefficients to the field $\mathfrak{R}(\E_a^\infty)$ of rational functions on $\E_a^\infty$,
we obtain the graded commutative algebra
 $$
\mathfrak{A}=\mathfrak{R}(\E_a^\infty)\ot_{\mathfrak{R}(\E_a^1)}\Gamma(S\tau)=\sum_{i=0}^\infty\mathfrak{A}_i,\quad
\text{where }\mathfrak{A}_i=\mathfrak{R}(\E_a^\infty)\ot_{\mathfrak{R}(\E_a^1)}\Gamma(S^i\tau)
 $$
(here and below by an abuse of notations
$\tau(a_\infty)=a_1\equiv\t_a$ is the lifted to $\E^\infty$ tautological bundle and $\t_a^*$ is its dual).

Consider the space $\Omega^1(\E_a^k)$ of rational 1-forms on $\E_a^k$ and
its subspace $\Omega^1_h(\E_a^k)$ consisting of the forms vanishing on $\Delta_k$.
The latter is generated by differentials of the rational 1st integrals of the distribution
$\Delta_k$, which by Rosenlicht's theorem are the rational invariants of
$G^k_a$ on $\E^k_a$. Both are vector spaces over $\mathfrak{R}(\E_a^k)$
and the factor-space $\Omega^1(\E_a^k)/\Omega^1_h(\E_a^k)$ can be identified with the
space of rational sections of $\Delta_k^*$. Extending the coefficients yields the space
 $$
\Gamma_t(\Delta_k^*)=\mathfrak{R}(\E_a^t)\ot_{\mathfrak{R}(\E_a^k)}
[\Omega^1(\E_a^k)/\Omega^1_h(\E_a^k)],\quad k\le t\le\infty.
 $$
Furthermore the distribution dual to the distribution $\varpi_k$ is given by
 $$
\varpi_k^*(a_k)=\op{CoKer}(d\pi^*_{k,k-1}:\Delta^*_{k-1}(a_{k-1})\to\Delta^*_k(a_k)),
 $$
and its module of rational sections with coefficients extended to the $t$-th jets is
 $$
\Gamma_t(\varpi_k^*)=\mathfrak{R}(\E_a^t)\ot_{\mathfrak{R}(\E_a^k)}
[\Omega^1(\E_a^k)/(\Omega^1(\E_a^{k-1})+\Omega^1_h(\E_a^k))].
 $$

Let $\pi_\E^*\varpi_k^*$ be the pullback to $\E_a^\infty$ of the system $\varpi_k^*$ on $\E^k_a$ 
(symbolic from the jet-level $l_0$ that is the number from Proposition \ref{prp4}). Define 
$\varpi^*_{[l_0]}=\sum_{k\ge l_0}\pi^*_\E\varpi_k^*$ (alternatively consider 
$\tilde\varpi^*=\sum_{k\ge0}\pi^*_\E\tilde\varpi_k^*$ the dual completed system).
Consider the vector space over $\mathfrak{R}(\E_a^\infty)$ 
 $$
\mathfrak{M}=\Gamma_\infty(\varpi^*_{[l_0]})=\sum_{k=l_0}^\infty\mathfrak{M}_k,\quad
\mathfrak{M}_k=\Gamma_\infty(\varpi_k^*)=
\Gamma_\infty(\Delta_k^*)/\Gamma_\infty(\Delta_{k-1}^*).
 $$
This $\mathfrak{M}$ is a graded module over $\mathfrak{A}$.
To see this notice that on the Zariski open set $\pi_{\infty,k}^{-1}(\E_k')\subset\E^\infty$
the set $\{\varpi_i,l_0\le i\le k\}$ is a symbolic system, meaning that
$\delta^*:\tau\ot\varpi_{i-1}^*\to\varpi_i^*$, $i>l_0$. Since the restriction map
(where the second space of sections of the restricted system is defined similar to the first space)
 $$\Gamma_\infty(\varpi_k^*)\to\Gamma_\infty(\pi_\E^*\varpi_k^*|\pi_{\infty,k}^{-1}(\E_k'))$$
is an isomorphism (because the value of a rational section is entirely determined by its values on a Zariski
open subset), we conclude that the morphism $\delta^*$ induces the rational map
$\mathfrak{A}_1\ot\mathfrak{M}_k\to\mathfrak{M}_{k+1}$ (this follows from evaluation at a generic point)
and hence by iteration the rational map $\mathfrak{A}_l\ot\mathfrak{M}_k\to\mathfrak{M}_{k+l}$
compatible with the algebra structure on $\mathfrak{A}=\sum\mathfrak{A}_l$.
This proves the claim.

Moreover the module $\mathfrak{M}$ is a quotient of the Noetherian graded $\mathfrak{A}$-module
$\mathfrak{R}(\E^\infty_a)\ot\Gamma(S_{[l_0]}\tau\ot\nu^*)$, where $S_{[l_0]}\tau=\oplus_{k\ge l_0}S^k\tau$
(due to the projection dual to the injective map $\varpi\hookrightarrow S\t^*\ot\nu$).
Consequently the $\mathfrak{A}$-module $\mathfrak{M}$ is Noetherian.
This means that there exists a basis $\omega_1,\dots,\omega_s$ in $\mathfrak{M}$
such that any element $\omega\in\mathfrak{M}$ is a linear combination
$\omega=\sum_1^sf_i\omega_i$ for some (unique) $f_i\in\mathfrak{A}$.

Let $\E''_a\subset\E_a^\infty$ be the Zariski open set of all points $a_\infty$, where the basis
elements $\omega_1,\dots,\omega_s$ are defined (= finite) and linearly independent (over the ring $S\t_a$).
Since the coefficients of these elements are rational functions,
this set is given by a finite jet condition, meaning there exists a number $l$ and
a Zariski open nonempty set $\E_{l,a}''$ such that $\E''_a=\pi_{\infty,l}^{-1}(\E_{l,a}'')$.
In what follows we also assume that the coefficients of $\omega_i$ are from the field $\mathfrak{R}(\E^l_a)$.

We claim that if the evaluation $\omega(a_\infty)$ is defined (= finite) at $a_\infty\in\E''_a$, then
all the coefficients of the decomposition $\omega(a_\infty)=\sum_1^sf_i(a_\infty)\omega_i(a_\infty)$
are also defined. Indeed, the opposite case means that we can decompose
$f_i=\sum_j\frac{p_{ij}}{q_{ij}}r_{ij}$, where $p_{ij},q_{ij}$ are polynomial functions on
$\E_a^\infty$, $r_{ij}$ are sections of $S\t_a$
nonvanishing at $a_\infty$, and $q_{st}(a_\infty)=0$ for some $s,t$
(as usual, the numerators $p_{ij}$ and denominators $q_{ij}$ have no common factors).

Let $Q$ be an irreducible factor of the polynomial $q_{st}$ such that $Q(a_\infty)=0$.
By the Hilbert zeros theorem $p_{st}$ does not vanish identically on the algebraic variety $Q=0$
(in finite jets: all the involved polynomials are defined on a finite jet level,
so Hilbert's theorem applies). Similarly the other numerators and denominators
either have the factor $Q$ or do not vanish identically on this variety.
Thus we can perturb the point $a_\infty\in\E''_a$ to a point $a'_\infty\in\E''_a$
such that among the irreducible factors of $p_{ij},q_{ij}$ only $Q$ vanishes at $a'_\infty$.
Let $k$ be the maximal degree of the factor $Q$ among all denominators $q_{ij}$.
Then the relation $\sum_{ij}Q^k\frac{p_{ij}}{q_{ij}}r_{ij}\omega_i=Q^k\omega$ evaluated
at $a'_\infty$ is a linear relation (nontrivial: not all coefficients of the left-hand side vanish)
among the generators $\omega_i(a_\infty')$.
This is prohibited at $a_\infty'\in\E''_a$, and this contradiction implies the claim.

Next, we claim that for
$\omega\in\Gamma_k(\varpi_k^*)=\Gamma_k(\Delta_k^*)/\Gamma_k(\Delta_{k-1}^*)$, $k\ge l$,
the coefficients of decomposition $\omega=\sum_1^s f_i\omega_i$ are rational functions on
$\E^k_a$. Indeed, both $\omega$ and $\omega_i$ satisfy this, so if some $f_i$ depend on a
jet-coordinate of order $>k$, we can Taylor-expand in this variable and get linear relations
among $\omega_i$. This contradicts to the choice of $\omega_i$ (it is a basis).

For every $a_\infty\in\E''_a$ and every $\omega_o\in\varpi(a_\infty)$ there is an element  $\omega\in\mathfrak{M}$ that evaluates to $\omega(a_\infty)=\omega_o$.
Let $\omega_o=\sum_1^sf_i(a_\infty)\omega_i(a_\infty)$ be the corresponding decomposition.
By the above, all coefficients are defined. Moreover, by the previous remark the
evaluation at $a_\infty$ boils down to evaluation at $a_k=\pi_{\infty,k}(a_\infty)$.
Therefore we conclude that $\varpi_l^*(a_\infty)$ generates all $\varpi_k^*(a_\infty)$ for $k\ge l$.

In other words, the first Koszul homology vanish in the range $k\ge l$
(in fact, it vanishes in the range $k\ge r$, where $r$ is the biggest degree of $\omega_i$).
By duality, this implies vanishing of the first Spencer cohomology
$H^{1,i}(\varpi(a_\infty))=0$ for $i\ge l$. This means precisely that
$\varpi_{k+1}=\varpi_k^{(1)}$ for all $a_\infty\in\E''_a$.

Finally the generating property is clearly $G$-invariant, so letting
$\E''=G\cdot\E''_a=\pi_{\infty,l}^{-1}(\E''_l)$
for $\E''_l=G^l\cdot\E''_{l,a}$ we obtain the required Zariski open set.
 \end{proof}

\smallskip

Thus we have proved that for a Zariski open (nonempty) set in $\E^\infty$ the symbolic system
$\varpi(a_\infty)_{\ge l_0}$ (or the completed symbolic system $\tilde\varpi(a_\infty)$) is stable.
By the Poincar\'e $\d$-lemma \cite{S$_2$} we conclude the following

 \begin{cor}\label{cor10}
There exist $l\in\Z_+$ and a Zariski open $\E''=\pi_{\infty,l}^{-1}(\E''_l)\subset\E$
such that the Spencer cohomology groups $H^{i,j}(\varpi)$, $i\ge l$, $j\ge0$, vanish $\forall$ $a_\infty\in\E''$.
 \end{cor}

\subsection{Existence of invariant derivations}\label{S23}

Recall that $\t_k=\pi_{k,1}^*\t$ at the point $a_k\in\E^k$ can be identified with the horizontal
plane $L(a_k)\subset T_{a_{k-1}}\E^{k-1}$, so we have two bundles $\t_k$ and $\t_k^*$ over $\E^k$.
Define the family of subspaces
 $$
\a_k^*(a_k)
=\{\hat{d}_{a_k}f|f\in\mathcal{A}_{k-1}\}\subset\t^*_k(a_k).
 $$
This is a sub-distribution in the distribution of evaluations of $G$-invariant
horizontal 1-forms on the Zariski open set $\E_k''$.
Pullback by the projection induces the inclusions $\a_k^*(a_k)\subset\a_{k+1}^*(a_{k+1})$
for $a_{k+1}\in F(a_k)$.
Thus we have the canonically defined limit
 $$
\a^*(a_\infty)=\lim_{k\to\infty}\a^*_k(a_k)=\cup\,\a^*_k(a_k).
 $$
This increasing flag of distributions on $\E^\infty$ stabilizes in finite jets:

 \begin{theorem}\label{P10+}
There exists a number $l$ and a Zariski open non-empty subset\footnote{To keep the notations
simple, we denote the required subset in the set of regular points by the same symbol $\E_l''\subset\E_l'$
(we also abuse the notations by re-defining $l$ several times).} $\E_l''\subset\E^l$ such
that $\a^*_l\subset\t^*_l|_{\E_l''}$ is a regular smooth sub-distribution and that rank of $\a^*_k$
at $a_k\in\pi_{k,l}^{-1}(\E_l'')$ is independent of $k\ge l$ and a choice of the point $a_k$.
 \end{theorem}

The proof uses the same arguments as in Theorem~\ref{trma} and so is omitted.
Define the following distributions (they are regular smooth on $\E_k''$):
 $$
\a_k^\vee(a_k)=\op{Ann}(\a_k^*(a_k))\subset\t_k(a_k)\quad\text{ and }\quad
\a_k(a_k)=\t(a_k)/\a_k^\vee(a_k)
 $$
Then we define the limits
 $$
\a^\vee(a_\infty)=\lim_{k\to\infty}\a_k^\vee(a_k)\quad\text{ and }\quad
\a(a_\infty)=\lim_{k\to\infty}\a_k(a_k).
 $$
Notice that the distribution $\a$ is dual to $\a^*$, but it is not
a sub-distribution in $\t$ in a natural way. Thus its sections can be represented by
horizontal vector fields ($\mathcal{C}$-fields) defined modulo the sections of $\a^\vee$.

It is important to notice that these latter sections by definition annihilate
the algebra $\A$ of differential invariants (they correspond to $\mathfrak{N}_0$
of Definition \ref{DfnDer}), so independently of the representative from
$\a^\vee$ such $\mathcal{C}$-fields map $\A$ to itself. In addition, these derivations have
filtration $+1$, i.e.\ map $\A_i\to\A_{i+1}$ for sufficiently large $i$, and so
produce new differential invariants from any given set.

Denote the stable rank of the distribution $\a_l|_{\E_l''}$ by
$s=\sup\dim\a(\a_\infty)$, this is also the rank of the distribution $\a^*$.

 \begin{theorem}\label{tHmT}
On an open dense subset $\E''_\infty\subset\E^\infty$ there
exist $s=\dim(\a)$ independent invariant derivations $\nabla_i:\A\to\A$,
$i=1,\dots,s$. Any other derivation can be expressed as a linear combination of these with
coefficients from $\A$.
 \end{theorem}

In this statement the algebra $\A$ of differential invariants is defined as $\mathfrak{P}_l(\E)^G$
with $l$ from Theorem \ref{P10+} increased by 1. Derivations $\nabla_i$ are rational
(in the fiber jet-variables) sections of the bundle $\a$ over the Zariski open set
$\E''_\infty=\pi_{\infty,l}^{-1}(\E''_l)$, the module of these sections is defined as in the proof
of Theorem~\ref{trma}.

\smallskip

 \begin{proof}
The natural pairing between $\a=\t/\a^\vee$ and $\a^*$ corresponds to derivations:
for $v\in\a$ and $\hat d f\in\a^*$ their inner product is $i_v\hat df=v(f)$, $f\in\A$.
This paring is non-degenerate by the duality (or definition of $\a^\vee$).

Thus any basis $\hat d f_1,\dots,\hat d f_s$ of $\a^*$
(with some choice of differential invariants $\{f_i\}_{i=1}^s$ in $\A$)
gives rise to a basis of $\a$ (at a generic point), which consists of invariant derivations
with rational functional coefficients.
The construction is similar to the Tresse derivatives from Section \ref{S14}.

Indeed, we have $f_1,\dots,f_s\in\A_l$ and can suppose that
$\op{rank}(\hat{d}f_1,\dots,\hat{d}f_s)=s$ on $\E''_l$
(otherwise we shrink this set by removing a Zariski closed subset).
Let us complete this collection to a horizontal basis by some functions $g_{s+1},\dots,g_n$,
namely we get $\op{rank}(\hat{d}f_1,\dots,\hat{d}f_s,\hat{d}g_{s+1},\dots,\hat{d}g_n)=n$ on
$\E''_l$ (again it is possible that the latter set should decrease).

Now we form the Tresse derivatives
$\hat\p/\hat\p f_1,\dots,\hat\p/\hat\p f_s,\hat\p/\hat\p g_{s+1},\dots,\hat\p/\hat\p g_n$,
consisting of the horizontal vector fields (derivatives) dual to the above basis of horizontal
differentials. While the $\mathcal{C}$-fields $\hat\p/\hat\p f_i$ depend on the choice of $g_j$,
the derivations $\nabla_i=\hat\p/\hat\p f_i\mod\a^\vee$ do not depend on it.
They map $\A$ to itself and can be represented by combinations of total derivatives with
coefficients being rational functions in jets of order $\le l+1$ (but this representation by horizontal
vector fields is non-canonical, they are defined modulo $\a^\vee$; see example in \S\ref{S42}).

Suppose that $\nabla\in\mathfrak{Der}(\E,\t)$ is any other invariant derivation. Since the
above set of derivations is maximal we can decompose $\nabla=I_1\nabla_1+\dots+I_s\nabla_s$.
Then $I_i=\nabla(f_i)\in\A$, and we are done.
 \end{proof}

 \begin{remark}\label{Tresse-always}
It follows from the proof of the above theorem, that the basis derivations
$\nabla_1,\dots,\nabla_s$ can be always chosen to mutually commute.
 \end{remark}

The derivations $\nabla_i$ from Theorem \ref{tHmT} do not coincide in general with
invariant derivatives, interpreted as $G$-stable horizontal vector fields on $\E^\infty$.
However, it is possible to identify the invariant derivations with invariant $\mathcal{C}$-fields
under some additional assumptions.
One occasion is this: If $G^k$ is reductive, then its stabilizer
at generic point $a_k$ is such (theorems 7.12 and 7.15 in \cite{PV}) and so there exists an invariant
complement to $\alpha^\vee$ (lifted to $L(a_k)$ at all $a_k\in\E^k$).

Another possibility is given by a condition from \cite{KL$_2$}
similar to Kumpera's hypothesis H$_3$ \cite{Kum}. If
 $$
\op{Reg}_k^2(\E,G)=\{a_k\in\E^k\,|\, \exists a_{k+1}\in\E^{k+1}: \Delta_k(a_k)\cap L(a_{k+1})=0\}
 $$
is open and dense in $\E^k$ (for $k=l$ and so for all larger $k$), then the number of
invariant derivatives (horizontal vector fields) is $n$.

We return to this question in Section \ref{S33}, where a
sufficient condition for equality $\mathfrak{D}(\E,\t)=\mathfrak{Der}(\E,\t)$
will be given in Theorem \ref{Thmn}.

\subsection{Symbolic cohomology of the orbit spaces}\label{S24}

Consider the space of orbits $\tilde{Y}^k=\E^k/G^k$. In general this space has a
complicated structure, but our action is algebraic and transitive on the base, so
$\tilde{Y}^k=\E^k_a/G_a^k$ is a rational quotient and the singularities are nowhere dense.
Moreover $Y^k=(\tilde{Y}^k)_0\subset\tilde{Y}^k$ is the geometric quotient
of the $G$-action on an open subset of $\E^k_a$ (subscript 0 denotes the regular part as in \S\ref{S21}).

Denote by $\mathfrak{d}_k$ the fiber of the projection $TY^k\to TY^{k-1}$.
This and the previous symbol spaces are united into the following commutative diagram
(all terms evaluated at respective points determined by a regular point $a_k\in\E_k'$)
 \begin{equation}\label{130300}
  \begin{CD}
 @. 0 @. 0 @. 0  \\
@. @VVV  @VVV @VVV @.\\
 0\!\! @>>> \varpi_k @>>> \Delta_k @>>> \Delta_{k-1} @>>> 0 \\
@. @V{\ll}VV  @VVV @VVV @.\\
 0\!\! @>>> g_k @>>>  T\E^k @>>> T\E^{k-1} @>>> 0 \\
@. @VVV  @VVV @VVV @.\\
 0\!\! @>>> {\frak d}_k @>>> TY^k @>>> TY^{k-1} @>>> 0   \\
@. @VVV  @VVV @VVV @.\\
 @. 0 @. 0 @. 0  \\
  \end{CD}
 \end{equation}
Its rows are exact by definition and since the last two columns are exact, the first column is exact too.
The morphism $\ll$ is given by the action of the Lie algebra sheaf $\mathcal{G}$ of the pseudogroup $G$.

 \begin{prop}\label{P9}
Letting $\T_j(a_k)=\{d_{a_k}f\,|\,f\in\mathcal{A}_j\}\subset T_{a_k}^*\E^k$ we have:
 $$
\mathfrak{d}_k=(\T_k/\T_{k-1})^*,
 $$
 \end{prop}

 \begin{proof}
The epimorphism $\T_k\to\mathfrak{d}_k^*$ is obvious. We have to prove that its kernel is
$\T_{k-1}$. Since only evaluations of the differentials are involved, we can change (in this proof only)
the algebra $\mathcal{A}_i$ to the field of rational differential invariants of order $i$,
which we denote by the same symbol.

Let us choose a transcendental basis $h_1,\dots,h_s$ of the field $\mathcal{A}_{k-1}$
and complete it to a basis $h_1,\dots,h_s,f_1,\dots,f_t$ of the field $\mathcal{A}_k$.
We want to prove $F\in \mathcal{A}_k$, $d_{a_k}F|_{g_k}=0$ $\Rightarrow$ $d_{a_k}F=d_{a_k}H$ for some
$H\in \mathcal{A}_{k-1}$ at a generic point $a_k\in\E^k$ (where the above functions are
defined and have independent differentials).

We can express $F=Q(h_1,\dots,h_s,f_1,\dots,f_t)$ for some algebraic function $Q$
(this can involve algebraic extensions). Then
$d_{a_k}F|_{g_k}=\sum Q_{f_i}(a_k)\cdot d_{a_k}f_i=0$ implies $Q_{f_i}(a_k)=0$.
Let $H=Q(h,c_k)\in\mathcal{A}_{k-1}$ be the function obtained by freezing the argument $f$:
$c_k=f(a_k)$. Then $d_{a_k}H=d_{a_k}F$.
 \end{proof}

\smallskip

With this representation, $\mathfrak{d}_k^*=\T_k/\T_{k-1}$ are interpreted as the space of symbols of
differential invariants (at regular points).
Then the Koszul homomorphism $g_k^*\ot\t_a\to g_{k+1}^*$
(obtained from the epimorphism $S^k\t_a\ot\nu^*_a\to g_k^*$)
induces the natural map
 \begin{equation}\label{tresse}
\delta^*:\mathfrak{d}_k^*\ot\t\to\mathfrak{d}_{k+1}^*,
 \end{equation}
which is interpreted as the symbol of the invariant derivation.

Let us now combine the exact 3-sequences of the first column (\ref{130300})
with the corresponding Spencer $\d$-complexes into the commutative diagram ($k\ge n+l$):
 \begin{equation}\label{130301}
  \begin{CD}
 @. 0 @. 0 @. 0  \\
@. @VVV  @VVV @VVV @.\\
 0\!\! @>>> \varpi_k @>{\d}>> \varpi_{k-1}\ot\t^* @>{\d}>> \varpi_{k-2}\ot\La^2\t^* @>{\d}>> \!\!\dots \\
@. @VVV  @VVV @VVV @.\\
 0\!\! @>>> \!\!g_k\!\! @>{\d}>>  \!\!g_{k-1}\ot\t^*\!\! @>{\d}>> \!\!g_{k-2}\ot\La^2\t^*\!\! @>{\d}>> \!\!\dots \\
@. @VVV  @VVV @VVV @.\\
 0\!\! @>>> {\frak d}_k @>{\d}>> {\frak d}_{k-1}\ot\t^* @>{\d}>> {\frak d}_{k-2}\ot\La^2\t^* @>{\d}>> \!\!\dots \\
@. @VVV  @VVV @VVV @.\\
 @. 0 @. 0 @. 0  \\
  \end{CD}
 \end{equation}

The cohomology of the bottom complex, induced by the map $\d$ dual to (\ref{tresse}),
will be called the symbolic cohomology of the orbit spaces
(its vanishing in the stable range implies stabilization of certain Spencer-like D-cohomology),
and will be denoted $H^{i,j}({\frak d})$ according to the bi-grading.

 \begin{theorem}\label{refpro}
There exists a number $l$ and a Zariski open subset $\E_l''\subset\E^l$
such that for all points $a_k$ with $\pi_{k,l}(a_k)=a_l\in\E_l''$
the $\d$-cohomology groups $H^{i,j}({\frak d})$ vanish in the range $i\ge l,j\ge0$.
 \end{theorem}

 \begin{proof}
When $l$ is at least the involutivity order for $\E$, i.e.\ $H^{i,j}(g)=0$ for $i\ge l,j\ge0$, then
the standard diagram chase (snake lemma) implies isomorphism of the $\d$-cohomology groups
$H^{i,j+1}(\varpi)=H^{i+1,j}({\frak d})$. If in addition $l$ satisfies the assumption of
Corollary \ref{cor10} and Theorem \ref{tHmT}, then we get $H^{i,j}({\frak d})=0$
in the range $i\ge l,j\ge0$ over the same Zariski open set $\E_l''$
(common for both Corollary \ref{cor10} and Theorem \ref{tHmT}).
 \end{proof}

\smallskip

Thus for $k\ge l$ the number $\dim{\frak d}_k$ grows polynomially
(this Hilbert polynomial will be discussed in \S\ref{S35}). Consequently,
the number of independent differential invariants of order $\le k$
(generators in $\mathcal{A}_k$) grows regularly, and this is the underlying
phenomenon for the Lie-Tresse theorem.


\subsection{An example of computations}\label{S25}

To illustrate the introduced notions let us consider the pseudogroup $G$ parametrized by 2 functions of 1 argument
and acting on $M=\R^3$ as follows:
 $$
g\cdot(x,y,z)=\left(X(x),y+Y(x),\frac{u}{X'(x)}\right).
 $$
This corresponds to the action of the (local) diffeomorphism group on $\R^2(x,y)$ preserving the
foliation $\{x=\op{const}\}$ and the length along these fibers, lifted to $\R^3(x,y,u)$ via the induced action on
densities on the base of the foliation: $u\,dx$.

If we restrict to $Y\equiv0$ this becomes the well-studied example by S.\,Lie, A.\,Tresse and A.\,Kumpera \cite{Kum}.
This example is rather simple\footnote{Though the Lie-Tresse-Kumpera pseudogroup
is not transitive (and so formally our theorems are not applicable),
this is the case where our theory still works.}: there is only one singular obit $\{u=0\}$ and the
algebra of invariants is generated by 1 differential invariant and two invariant derivations.

In the presence of $Y$ the situation is more interesting. We will consider the action on the submanifolds
transversal to the fibers of the projection $\pi:\R^3(x,y,u)\to\R^2(x,y)$, so we study the differential
invariants in the space $J^\infty\pi$. The Lie algebra sheaf of $G$ is
 $$
\mathcal{G}=\{\a(x)\p_x+\b(x)\p_y-\a'(x)u\p_u:a,b\in C^\infty(\R)\}.
 $$

The pseudogroup $G$ is determined by the algebraic differential equations
 $$
\frac{\partial X}{\partial y}=\frac{\partial X}{\partial u}=
\frac{\partial Y}{\partial u}=\frac{\partial U}{\partial y}=0,\
\frac{\partial Y}{\partial y}=1,\ \frac{\partial U}{\partial u}\cdot\frac{\partial X}{\partial x}=1,\
u\frac{\partial U}{\partial u}=U,
 $$
and so is algebraic. Hence the lifted action in $J^{k}\pi$ are algebraic for all $k>0$.

It is easy to see that this $G$-action on $J^0\pi=M$ has two orbits:
regular (open) $\E'_0=\{u\neq 0\}$ and singular (closed) $S_0=\{u=0\}$
(the function $u$ is a relative invariant of the action).
The quotient $J^0\pi/G$ is the two-point Sierpi\'{n}ski space, while the geometric quotient
$\E'_0/G$ is the one-point space.

Denote by $\op{St}(a_k)$ the stabilizer of the pseudogroup $G_a$ at the point $a_k$, and by
$\op{Is}(a_k)=\op{Lie}(\op{St}(a_k))\subset\mathcal{G}_a$ the isotropy algebra.
Using the jet-notations $X_i$ for $X^{(i)}$ and $Y_i$ for $Y^{(i)}$,
we have for $a\in\E'_0$: $\op{St}(a)=\{X=x,Y=0,X_1=1\}$, and the action of $\op{St}(a)$ on $F(a)$ is:
 $$
(u_x,u_y)\mapsto (U_x,U_y)=(u_x-X_2u-Y_1u_y,u_y).
 $$
Thus all orbits are regular, and there is 1 first order differential invariant
 $$
I_1=\frac{u_y}u.
 $$
Next $\op{St}(a_1)=\{X=x,Y=0,X_1=1,X_2=-I_1Y_1\}$ and its action on $F(a_1)$ is:
 $$
(u_{xx},u_{xy},u_{yy})\mapsto(U_{xx},U_{xy},U_{yy})=
(u_{xx}+...-X_3u,u_{xy}-Y_1uI_2,u_{yy}),
 $$
where by ... we denote inessential terms, and
 $$
I_2=D_yI_1
 $$
is another invariant. The orbits are regular iff $I_2\neq0$, i.e.\ $\mathcal{E}_2'=\{u\neq0,I_2\neq0\}$.
On the singular stratum $S_2=\{u\neq0,I_2=0\}$ dimensions of the orbit drop.

In higher jets the behavior stabilizers: $\mathcal{E}_k'=\pi^{-1}_{k,2}(\mathcal{E}'_2)$ and
$\op{St}(a_k)=\{X=x,X_1=1,X_2=\dots=X_k=0,Y=Y_1=\dots=Y_{k-1}=0,X_{k+1}=-I_1Y_k\}$,
and the action of the stabilizer on $F(a_k)$ is
 $$
u_{i,k+1-i}\mapsto U_{i,k+1-i}=\left\{
\begin{array}{l} u_{k+1,0}+...-X_{k+2}u,\ i=k+1\\
u_{k,1}-Y_kuI_2,\ i=k\\ u_{i,k+1-i},\ i<k\end{array}\right.
 $$
(with jet-notations $u_{i,j}$ for $\partial_x^i\partial_y^ju$).
This can be also obtained from the Lie algebra action, since $\op{Is}(a_k)$ acts on $F(a_k)$
by shifts along the fields $\partial_{u_{k+1,0}},\partial_{u_{k,1}}$.

The invariant derivations are
 $$
\nabla_1=D_y(I_1)D_x-D_x(I_1)D_y,\ \nabla_2=D_y
 $$
and we get the following differential invariants according to their orders:
 \begin{gather*}
I_1; \ \ I_2=\nabla_2I_1; \ \
I_{3a}=\nabla_1I_2,\ I_{3b}=\nabla_2I_2; \\
I_{4a}=\nabla_1I_{3a},\ I_{4b}=\nabla_1I_{3b},\ I_{4c}=\nabla_2I_{3b}; \ \dots
 \end{gather*}
This algebra separates the regular orbits of $G$.

The symbolic system from \S\ref{S22} is:
 $$
\varpi_0=\langle\partial_u\rangle,\quad
\varpi_1=\langle dx\otimes\partial_u\rangle,\quad
\varpi_k=\langle dx^{k}\otimes\partial_u,dx^{k-1}dy\otimes\partial_u\rangle,\ k\ge2.
 $$
The completed symbolic system differs only in $\tilde\varpi_2=\tau^*\ot\nu$.
In general, this differs from the symbol of the pseudogroup $\mathfrak{g}_k\subset S^kT^*M\otimes TM$,
but in our example we can identify $\varpi_k\simeq\g_k$ for $k\ge2$ ($\tilde\varpi_k\simeq\g_k$ for $k\ge1$).

Notice that for $\xi=\partial_y$ we have $\delta_{\xi}:\varpi_2\not\to\varpi_1$, and $\varpi_k$
becomes the symbolic system starting from the jet-level 2 over $\E_2''=\E_2'$.
Moreover $\varpi_{k+1}=\varpi_k^{(1)}$ for $k\ge2$; in fact, the stabilization level for both Proposition \ref{prp4}
and Theorem \ref{trma} is $l=2$ in this case. The $\delta$-complex for this symbolic system has the form
 \[
0\to\varpi_{k+2}\overset{\delta}\to\varpi_{k+1}\otimes\tau^*\overset{\delta}\to\varpi_k\otimes\Lambda^2\tau^*\to0,
 \]
where $\tau^*=\mathbb{R}^2(dx,dy)$ and so (we omit $\otimes\partial_u$ in both sides)
 \begin{gather*}
\delta(dx^k\otimes\omega)=k\,dx^{k-1}\otimes dx\wedge\omega, \\
\delta(dx^{k-1}dy\otimes\omega)=dx^{k-1}\otimes dy\wedge\omega
+(k-1)\,dx^{k-2}dy\otimes dx\wedge\omega.
  \end{gather*}
Therefore the symbolic system is involutive: $H^{k,i}(\varpi)=0$, $k\ge2$, $0\le i\le 2$.
For the completed symbolic system $\tilde\varpi$ the only non-trivial cohomology are:
 $$
H^{0,0}(\tilde\varpi)=\R,\quad H^{1,1}(\tilde\varpi)=\R.
 $$

In $\mathcal{E}_k'$, i.e.\ when $I_2\neq0$, we obviously have $\alpha_k^*=\tau^*$, so
$\alpha^\vee_k=0$ and $\alpha_k=\tau(a_k)$ for $k\ge2$ in the notations of \S\ref{S23}.

Next, for the objects of \S\ref{S24} we have $\Theta_k=\langle dI_{i\z}:i\le k\rangle$ and
 \begin{gather*}
\mathfrak{d}_0=0,\qquad \mathfrak{d}_1=\langle dy\otimes\partial_u\rangle,\\
\mathfrak{d}_k=\langle dx^{k-2}dy^2,dx^{k-3}dy^3,\dots,
dx\,dy^{k-1},dy^k\rangle\otimes\partial_u,\ k\ge2
 \end{gather*}
(these shall be considered as quotient spaces, not subspaces of $g_k=S^k\t^*\ot\nu$).
The middle complex in diagram (\ref{130301}) is exact, and the only non-trivial symbolic cohomology
of the bottom complex is
 \[
H^{1,0}(\mathfrak{d})=\R=H^{1,1}(\tilde\varpi).
 \]

\section{Differential invariants}\label{S3}

In this section we prove our main results on the structure of the algebra $\A=\mathfrak{P}_l(\E)^G$
of global rational-polynomial differential invariants of an algebraic pseudogroup $G$ action.
We also prove finiteness of other invariant quantities (infinite-dimensional in the usual sense)
in the spirit of the Lie-Tresse theorem.

\subsection{Stabilization of singularities and separation of orbits}\label{S31}

Let us begin by proving the affine property for the projections of orbits.
Let $l$ be the maximal among the following integers: the stabilization jet level $l$
from Theorem \ref{trma}, the involutivity levels of $\E$ and $G$, the number of Theorem \ref{refpro}
and the stabilization level of the distribution $\a^*$ from Theorem \ref{P10+} increased by 1.

 \begin{prop}\label{padepa}
With jet-level $k$ starting from this number $l$, the regular orbits
$G^{k+1}\cdot a_{k+1}\subset\E_{k+1}''$ are affine bundles over the regular orbits
$G^k\cdot a_k\subset\E_k''$.
 \end{prop}

Notice that the action of $G^{k+1}$ is affine in the fibers
$F(a_k)\cap\E^{k+1}$, but this is not enough for the
above claim (consider for instance the orbits of a linear Lie group action on the vector space).

\smallskip

 \begin{proof}
The invariant derivations from Theorem \ref{tHmT} act
$\nabla_i:\mathcal{A}_k\to\mathcal{A}_{k+1}$ for $k\ge l$
(note that if the coefficients of $\nabla$ are of order $p$, then $\nabla(f)$ has order $\ge p$).
By Theorem \ref{refpro} $H^{k,0}({\frak d})=0$ and so
$\delta:{\frak d}_{k+1}\to{\frak d}_k\ot\t^*$ is a monomorphism or equivalently
$\delta^*:{\frak d}^*_k\ot\t\to{\frak d}^*_{k+1}$ is an epimorphism.

Recall from \S\ref{S24} that this map (\ref{tresse}) is the symbol of invariant derivation operation.
From \S\ref{S23} we know that $\delta^*:{\frak d}^*_k\ot\a^\vee\to0$, whence we obtain the epimorphism
$\delta^*:{\frak d}^*_k\ot\a\to{\frak d}^*_{k+1}$.

Thus starting from the level $l$, when the amount of invariant derivations is stabilized, given a maximal
collection of differential invariants of order $\le k$ the invariant derivations applied to them
generate the symbols of all differential invariants of order $k+1$ at every point $a_{k+1}\in\E_{k+1}''$
(if there exists a differential invariant of order $(k+1)$ independent of the collection $\cup_i\nabla_i(\mathcal{A}_k)\cup\mathcal{A}_k$, then one can
obtain a new independent derivation 
and the collection $\nabla_i$ is not maximal).

Now given a basis $\{f^k_j\}$ in $\mathcal{A}_k$ the symbols of $f^k_j$ and of $\nabla_i(f^k_j)$
generate the symbols of the invariants from $\mathcal{A}_{k+1}$ at the regular points.
In other words, at any point $a_{k+1}\in\E_{k+1}''$ the differentials $df_j^{k+1}$
can be expressed via the differentials  $df^k_j$ and $d\nabla_i(f^k_j)$.
Since the latter are affine in the coordinates on the fibers $F(a_k)\cap\E^{k+1}$
and the orbits are the integral leaves of the kernel distribution of the differentials
of the invariants, the claim follows.
 \end{proof}

\smallskip

Now we collect the obtained knowledge to prove our first main result.

\smallskip

 \begin{proof}[Proof of Theorem \ref{Thm1}]
Let $l$ be as chosen before Proposition \ref{padepa} (this number is taken to exceed
the order of involutivity of the equation $\E$). Since the action of $G$ is transitive
on the base, all orbits in $\E^k$ project onto $M$.
So to study the space of orbits on the level of $k$-jets $\tilde{Y}^k$
it is enough to restrict to one fiber $\E^k_a$ and the action of $G^k_a$ on it:
$Y^k=(\tilde{Y}^k)_0=\E_k''\cap\pi_k^{-1}(a)/G^k_a$.

Points of this space are orbits $G^k_a\cdot a_k$ and by Proposition \ref{padepa}
the fibers of the bundle $G^{k+1}_a\cdot a_{k+1}\to G^k_a\cdot a_k$ are affine.
Since these are the orbits of the action in the affine fibers
$\E^{k+1}\cap\pi_{k+1,k}^{-1}(a_k)$, the quotient -- fibers of the
projection $Y^{k+1}\to Y^k$ are affine, i.e.\ this latter is an affine
bundle for all $k\ge l$.

In particular, there are no (higher) singularities over the set of regular points
$\E_l''=\E^l\setminus S_l$. All the singular orbits belong to the stratum $\pi_{\infty,l}^{-1}(S_l)\subset\E^\infty$, which has finite codimension.
The Zariski closed variety $S_l=\E^l\setminus\E_l''$ is the usual locus of
singularities for algebraic actions, see \cite{SR}, to which we added the set of
singularities of the (rational) invariant derivations and differential invariants.

Proposition \ref{RSL} provides the geometric quotient $\E_k''/G^k\simeq Y^k$ for all $k$.
For $k\ge l$ the singularities are stabilized to belong to $\pi_{k,l}^{-1}(S_l)$,
so the quotients form a tower of affine bundles $Y^{k+1}\to Y^k$ $\forall k\ge l$, and
the claim is proved.
 \end{proof}

\smallskip

Note that the above arguments not only prove Theorem \ref{Thm1}, but also the
separation property of Theorem \ref{Thm2} for rational differential invariants.

Indeed, the differential invariants can be chosen affine in the fibers of $Y^{k+1}\to Y^k$, $k\ge l$,
so they separate the orbits in higher jets provided the orbits of $G^l_a$-action in $\E^l_a$ are
separated by the rational invariants. But this latter is guaranteed by the Rosenlicht's theorem
(we again use separation of complex generic orbits by complex rational invariants,
and then take real and imaginary parts of the latter to separate the real orbits).

Actually, by Rosenlicht's theorem the field of rational differential invariants $\mathfrak{R}(\E''_k)^G$
coincides with the field of rational functions on the quotient variety $\mathfrak{R}(Y^k)$.
Yet for $k>l$ we can separate the $G$-orbits in $\E''_k$ by a smaller algebra of
differential invariants, namely by $\A=\mathfrak{P}_l(\E)^G$ due to the above affine property.

\subsection{Global Lie-Tresse theorem}\label{S32}

Now we prove the second main result. Let $l$ be the same integer as chosen
at the beginning of Section \ref{S31}.

\smallskip

 \begin{proof}[Proof of Theorem \ref{Thm2}]
By the choice of $l$ there exists on the level $l$ a maximal collection
of independent rational invariant derivations $\nabla_1,\dots,\nabla_s:\A\to\A$,
where $s=\op{codim}\a^\vee$ is the number from Theorem \ref{tHmT}.
They act by shifting the filtration $\A_k\subset\A$ by $+1$ for $k\ge l$. More generally,
for a multi-index $J$ the iterated derivations map the algebra of differential invariants
to itself $\nabla_J:\mathcal{A}_k\to\mathcal{A}_{k+|J|}$, $k\ge l$.

Since the geometric quotient $Y^l=\E''_l\cap\pi_l^{-1}(a)/G^l_a$ is finite-dimensional,
the transcendence dimension $(t-1)$ of $\mathfrak{R}(\E''_l\cap\pi_l^{-1}(a))^G\simeq\mathfrak{R}(Y^l)$
is finite. Then by Rosenlicht's theorem there exist rational differential invariants
$I_1,\dots,I_{t-1}$ of order $l$ generating a subfield in $\mathfrak{R}(\E''_l\cap\pi_l^{-1}(a))^G$
such that any other element in this field is obtained by a (finite) algebraic extension.
By the theorem on primitive element\footnote{Over characteristic 0 this theorem states that a finite
extension $E$ of a field $K$ is generated by one element $E=K(\alpha)$, see e.g.\ \cite[V.\S4: Theorem 4.6]{La}.}
there exists an element $I_t$ of this field
(that is a rational differential invariant) such that the collection $I_1,\dots,I_{t-1},I_t$ rationally generate
the whole field $\mathfrak{R}(\E''_l\cap\pi_l^{-1}(a))^G$.

We claim that for $k\ge l$ all differential invariants of order $k$ are polynomial functions of $\nabla_J(I_j)$
with rational coefficients of the invariants $I_j$, where the length $|J|\le k-l$.
We shall prove this by induction on $k$ with the base of induction $k=l$ being trivially true.
Suppose the assumption holds on the level $k$ and consider the differential invariants of order $k+1$.

By Proposition \ref{padepa} the orbits of the action of $G^{k+1}$ on $\E''_{k+1}$ are affine
bundles over the orbits of $G^k$ on $\E''_k$.
Thus $\mathcal{A}_{k+1}$ is generated by differential invariants that are affine in the
coordinates on the fibers $F(a_k)$. The set $\mathcal{A}_k\cup\{\nabla_i(\mathcal{A}_k)\}_{i=1}^s$
also consists of the invariants affine in the same coordinates, and the first set is functionally
dependent on the second by the vanishing of symbolic cohomology of the orbit spaces
$H^{k,0}({\frak d})=0$ (see the proof of Proposition \ref{padepa}).

Denote by $I^k_j$ the basis of differential invariants we obtained on the jet level $k$.
Then any differential invariant $f\in\mathcal{A}_{k+1}$, which is affine in the
$\pi_{k+1,k}$-fiber variables, can be written as
 $$
f=h_0+h_1\,I^{k+1}_1+\dots+h_{r_k}\,I^{k+1}_{r_k},
 $$
where the set $\{I^{k+1}_b=\nabla_{i_b}(I^k_{j_b}),1\le b\le r_k\}$
is a rational basis on $\E''_{k+1}$ chosen among the invariants $\nabla_i(I_j^k)$ of order $k+1$
(with some indices $i_b,j_b$) and $h_s$, $0\le s\le r_k$, are some functions on $k$-jets.

Now if $X\in\mathcal{G}$ is a vector field from the Lie algebra sheaf of $G$ and $X^{(k)}$ its $k$-jet lift,
then differentiating the above linear relation on $f$ along $X^{(k+1)}$ we get
 $$
L_{X^{(k)}}(h_0)+\sum_{b=1}^{r_k} L_{X^{(k)}}(h_b)I^{k+1}_b=0,
 $$
so linear independence of $I^{k+1}_b$ implies that $h_s$ are differential invariants.

Similarly, any $f\in\mathcal{A}_{k+1}$, which is a polynomial of degree $q_k$
in the $\pi_{k+1,k}$-fiber variables, can be expressed as
 \begin{equation}\label{repdi}
f=\sum_{|\z|\le q_k}h_{\z}I^{k+1}_\z,
 \end{equation}
where we sum by multiindex $\z=(i_1,\dots,i_{r_k})$ and
$I^{k+1}_\z=\prod_{c=1}^{r_k}(I^{k+1}_c)^{i_c}$. Again $h_\z\in\mathcal{A}_k$
are differential invariants. Substituting the expression of
$h_\z$ via the basic invariants $I_j$ on $\E''_l$ and their invariant derivatives,
given by the induction assumption,
we express $f$ as a polynomial of the invariant derivatives of $I_j$ up to order $(k+1-l)$,
with coefficients being rational functions of $I_j$
(we remark that while the differential invariants can be chosen
affine in the highest derivatives, their behavior with respect to other jets of
order $>l$ is inevitably polynomial).

Now notice that the commutators of the invariant derivations are invariant derivations, so we
can decompose
 \begin{equation}\label{MCe}
[\nabla_i,\nabla_j]=\sum_{k=1}^s \varrho_{ij}^k\nabla_t,
 \end{equation}
where $\varrho_{ij}^k$ are differential invariants of order $\le l+1$. Therefore
any composition $\nabla_{i_1}\cdots\nabla_{i_q}(I_j)$ can be algebraically expressed via
$\nabla_J(I_j)$ with an ordered multi-index $J$. This yields algebraic expression of
differential invariants in $\mathcal{A}_{k+1}$ via $\nabla_J(I_j)$,
which is polynomial in the entries with $|J|>0$ and rational in $I_j$, and so
fulfils the induction step.

This finishes the proof of finiteness property in the Lie-Tresse sense. The separation
property is also proved by induction. Indeed, the proof of separation
by rational invariants from the end of \S\ref{S31}, works also for differential invariants
rational in order $\le l$ and polynomial in order $>l$.
This proves Theorem \ref{Thm2}.
 \end{proof}

\bigskip


 \begin{remark}
In the proof we have chosen a collection $I_1,\dots,I_t$, which taking into account
the invariant derivations $\nabla_1,\dots,\nabla_s$ is often excessive. There exists a minimal
collection of differential invariants and invariant derivations generating the algebra $\mathcal{A}$,
but the structure of this set is yet unknown. See some partial results by P.\,Olver and co-authors
in \cite{Ol$_3$}.

We have an alternative: either to represent
$\mathfrak{P}_l(\E)^G$ with the minimal number of basic invariants and derivations, or
to reduce $S$ in size (by decreasing the locus of dependencies for generators of the algebra
of differential invariants).
 \end{remark}

As noticed in the Introduction, the global Lie-Tresse theorem implies the familiar
micro-local Lie-Tresse theorem. Indeed for an open set $U_\infty=\pi_{\infty,l}^{-1}(U_l)$
($U_l\subset\E''_l$ needs not be $G$-invariant, this is a feature of the micro-local analysis)
restriction of the collection $\mathcal{A}_\infty$ of rational-polynomial invariants
form a basis of differential invariants. On every finite level $k\ge l$ a differential invariant
of order $k$ is functionally dependent on a basis in $\mathcal{A}_k$ and so can be expressed through it
by the implicit function theorem.

We can even make this expression without shrinking the neighborhood $U_k$
(standard in the application of implicit function theorem).
Indeed, the function constant on the orbits can be expressed by the coordinates on the
geometric quotient $Y_k$ (or its part corresponding to $U_k$) by the categorical property
discussed in \S\ref{S31} (the geometric quotient is a categorical quotient).

\subsection{Invariant derivations and other geometric structures}\label{S33}

The Lie algebra of rational invariant derivations $\mathfrak{Der}(\E,\t)$ is finitely
generated as an $\A$-module, namely we have by Theorem \ref{tHmT}:
 $$
\dim_{\A}\mathfrak{Der}(\E,\t)=s.
 $$

We shall discuss a relation of this to the Lie algebra $\mathfrak{D}(\E,\t)$ of invariant derivatives.
We first introduce the latter as a module of sections of a distribution over $\A$ in the spirit of Section \ref{S23}.

Let $\a^k_H(a_k)\subset\t_k(a_k)$ be the subspace of vectors invariant with respect to the stabilizer
$G_{a_k}$ of the point $a_k\in\E^k$. Since $\t_k\simeq\t$ via the natural projection $\pi_k$,
we have the following: $a_{k+1}\in F(a_k)$ implies $\a_H^{k+1}(a_{k+1})\supset\a_H^k(a_k)$.
Thus we have the canonically defined limit
 $$
\a_H(a_\infty)=\lim_{k\to\infty}\a_H^k(a_k)=\cup\,\a_H^k(a_k).
 $$
Moreover the stabilization occurs on a finite jet-level.

 \begin{prop}\label{P10}
There exists a number $l$ and a Zariski open non-empty subset $\E_l''\subset\E^l$ such
that $\a_H^l\subset\t_l|_{\E_l''}$ is a regular smooth sub-distribution and that rank of $\a_H^k(a_k)$
is independent of $k\ge l$ and the point $a_k\in\pi_{k,l}^{-1}(\E_l'')$ .
 \end{prop}
The proof uses the same arguments as in Theorem \ref{trma} and so is omitted.
Denote the stable rank of the distribution $\a_H^l|_{\E_l''}$ by
$s_H=\sup\dim\a_H(\a_\infty)$.

 \begin{cor}\label{poi}\label{Cr11}
There exists $s_H=\dim\a_H$ invariant derivatives of the $G$-action
generating the module $\mathfrak{D}(\E,\t)$ of invariant derivatives over
the algebra $\mathcal{A}$.
 \end{cor}

 \begin{proof}
Elements of the $\A$-module $\mathfrak{D}(\E,\t)$ are the rational sections
of the distribution $\a$ over $\E_\infty''$. Existence of the basis of
$s_H=\dim_\A\mathfrak{D}(\E,\t)$ rational invariant derivatives follows similarly to
Theorem \ref{tHmT}.

Alternatively, consider the $G^k$-action on the bundle $\t_k$ over $\E^k$, which we denote by
$\E^k\times\t$. It is equivalent to the $G^k_a$-action on $\E^k_a\times\t_a$ for any point $a\in M$.
This latter is algebraic and so by Rosenlicht's theorem there is a geometric quotient
 $$
(\E^k\times\t_k\setminus\hat S_k)/G^k=(\E^k_a\times\t_a\setminus\hat S_k\cap(\E^k_a\times\t_a))/G^k_a
\simeq {\hat Y}^k.
 $$
As before the singularity locus $\hat{S}_k$ stabilizes meaning that $\pi_{k,l}:\hat S_k\to\hat S_l$ for $k>l$
(in fact, we occasionally obtain that $\hat S$ fibers over the singularity locus $S$
for the differential invariants).

There is the natural algebraic morphism ${\hat Y}^k\to Y^k$. Dimension of its fibers corresponds
to the number of invariant derivatives on the level $k$, so this is an integer
between 0 and $n$ that (non-strictly) monotonically grows with $k$. Clearly it stabilizes at the number $s_H$
and the claim follows.
 \end{proof}

\bigskip

We would like to have a sufficient condition for $s_H=n$.
For this we introduce the following notion.

 \begin{dfn}
An equation $\E$ is called {\em ample} if for almost all $a_k\in\E^k$
($k$ starts with the order of involutivity) no irreducible component of
the characteristic variety $\op{Char}^\CC(\E,a_k)\subset\mathbb{P}^\CC\t_a^*$ belongs to one hyperplane.
 \end{dfn}

Notice that in this definition we can restrict to $k$ being the (maximal) order of $\E$,
since prolongation does not change the characteristic variety.

In the case $\E^k=J^k$ (no constraints on $N$) the condition of ampleness clearly holds.
Moreover, ampleness is the condition of general position for the projective variety
$\op{Char}^\CC(\E,a_k)$ of fixed dimension provided it is nonempty;
the opposite case reduces to algebraic actions of Lie groups on
finite-dimensional manifolds.

 \begin{prop}\label{P12}
If $\E$ is ample, then the map $\d_v=i_v\circ\d:g_{k+1}\to g_k$ is onto
for every non-zero vector $v\in\t$ ($k$ starts with the order of involutivity).
 \end{prop}

 \begin{proof}
By \cite{S$_2$} $\op{Char}^\CC(\E)$ is the support of the symbolic module $\mathcal{M}_\E=g^*$ over $S\t$
and every element $q_k\in g_k^\CC$ is a finite linear combination $q_k=\sum\ll_ip_i^k\ot\zeta_i$ with $\ll_i\in\CC$, $p_i\in\op{Char}^\CC(\E)$, $\zeta_i$ from the kernel sheaf.
Given $v\in\t\setminus0$ a generic $q_k$ has characteristic covectors $p_i\notin v_\CC^\perp$
in this sum (by ampleness).
Then for $q_{k+1}=\sum\frac{\ll_i}{k+1}\langle p_i,v\rangle^{-1}p_i^{k+1}\ot\zeta_i\in g_{k+1}^\CC$
we have $\delta_v(q_{k+1})=q_k$, whence $\d_v(g_{k+1}^\CC)=g_k^\CC$. The claim
for the real case follows.
 \end{proof}

 \begin{theorem}\label{Thmn}
Assume that $\E$ is ample and that the number of functionally independent $G$-differential
invariants is infinite. Then there exist $n$ independent invariant derivatives
$\nabla_1,\dots,\nabla_n\in\mathfrak{D}(\E,\t)$.
 \end{theorem}

In the case\footnote{Here by $\dim\mathcal{A}$ we understand the maximal number of
functionally independent differential invariants (recall that any function of differential invariants
is a differential invariant itself, so the dimension of $\mathcal{A}$ as a vector space is
either 1 or $\infty$).}
$\dim\mathcal{A}<\infty$, the $G$-action is eventually transitive
(orbits have finite codimension in $\E^\infty$); some such systems were investigated
in \cite{KL$_2$} (in fact it is enough to assume $\dim\mathcal{A}>n$ in the theorem).

\smallskip

 \begin{proof}
We claim that $s=\dim\a=n$ under the assumptions of the theorem.

Assume the opposite $s<n$.
Then for some independent differential invariants $f_i$, $1\le i\le s+1$, we have:
$\oo_s=\hat d f_1\we\dots\we\hat d f_{s-1}\we\hat d f_s\ne0$ at a generic point $a_{k+1}\in\E''_{k+1}$,
and $\oo_s\we\hat d f_{s+1}=0$. Moreover by the assumptions for all $i\le n$:
$\hat d f_1\we\dots\we\check{d f_i}\we\dots\we\hat d f_{s+1}=\z_i\oo_s$, where $\z_i$ are
some differential invariants.

Let $\z_i^0=\z_i(a_{k+1})$ be the (constant) values. Let $k\ge l$.
Changing $f_{s+1}$ to another differential invariant
$f_{s+1}-\sum_{i=1}^s(-1)^{s-i}\z_i^0f_i$ we get the following equalities at $a_{k+1}$:
$\hat d f_1\we\dots\we\check{d f_i}\we\dots\we\hat d f_{s+1}=0$ for $i\le n$.

We can suppose (by increasing $k$ and using $\dim\mathcal{A}=\infty$) that
$df_{s+1}|g_k\ne0$. Then by Proposition \ref{P12} there exists $\theta\in g_{k+1}$ such that
$\delta_v\theta\not\in\op{Ker}(df_{s+1}|_{g_k})$ for generic $v\in\t$.
Moreover, since our claim concerns the ranks we can work over $\CC$
and so choose $\theta=p^{k+1}\ot\zeta$, where the covector $p\in\t_\CC^*$
is characteristic and $\zeta$ is from the kernel sheaf (compare the proof of Proposition \ref{P12}).

By a small perturbation we can achieve $p|_{\op{Ker}\oo_s}\ne0$ (because $\E$ is ample)
and we also have $\rho=df_{s+1}(\bar\theta)\ne0$ for $\bar\theta=p^k\ot\zeta\in g_k$ by the above
assumption.
Finally we can change the differential invariants $\{f_i\}_{i=1}^s$ in such a way that the
above constraints hold, and in addition $df_i(\bar\theta)=0$, $1\le i\le s$, at the point $a_{k+1}$.

Let us now deform the point $a_{k+1}\in\E^{k+1}$ to $a_{k+1}^\epsilon=a_{k+1}+\epsilon\,\theta$
($\theta\in S^{k+1}\t^*\ot\nu$ is as above and we use the affine structure in the fibers).
Since the set $\E_{k+1}''$ is Zariski open, the point $a_{k+1}^\epsilon$ still belongs to this set
for small $\epsilon$.
Then the horizontal $(s+1)$-form $\oo_s\we \hat d f_{s+1}$ at the point $a_{k+1}^\epsilon$ equals
 $$
\hat d f_1\we\dots\we\hat d f_s\we\hat d f_{s+1}=\epsilon\rho(k+1)\,\oo_s\we p\ne0
 $$
for $\epsilon\ne0$. This contradicts the choice of $s$. Thus $s=n$.

Consequently there exist $n$ differential invariants
$f_1,\dots,f_n$ satisfying the condition $\hat d f_1\we\dots\we\hat d f_n\ne0$
in an open dense neighborhood $U\subset\E^\infty$,
which implies existence of $n$ invariant derivatives
$\nabla_i=\hat\p/\hat\p f_i$, $i=1,\dots,n$, by the construction \`a la Tresse, see \S\ref{S14}.
 \end{proof}

\smallskip

 \begin{cor}
Under the above assumptions we have $\mathfrak{Der}(\E,\t)=\mathfrak{D}(\E,\t)$. \qed
 \end{cor}

In general we can have $s_H=\dim\a_H<n$ (see example in \S\ref{S42}).
To see a relation between this number and $s$, let us recall a subbundle $\a_k^\vee\subset\t_k$
from Section \ref{S23}.
By the proof of Theorem \ref{Thmn} elements of $\a^\vee_k$ annihilate the characteristic variety $\op{Char}^\CC(\E,a_k)$ at every point $a_k$. These vectors act trivially on the symbolic module
of the equation $\E$:
 $$
\a^\vee\subset\{v\in\t:v^\perp\supset\op{Char}^\CC(\E)\}
 $$
(in other words, $v$ to the right are all non-regular elements for the module $g^*$).
In addition, (rational) sections of $\a^\vee$ act trivially on the algebra $\A$.

Every invariant derivative is also an invariant derivation. This morphism $\iota$
leads to the exact sequence
 $$
0\to\a_H^\vee\longrightarrow\a_H\stackrel{\iota}\longrightarrow\a,
 $$
where $\a_H^\vee=\a_H\cap\a^\vee$ is the kernel of $\iota$. Thus if we denote
$s_H^\vee=\dim\a_H^\vee$ the maximal number of independent invariant derivatives acting trivially
on differential invariants, we conclude:
 $$
s_H-s_H^\vee\le s.
 $$

Similar as invariant derivatives (and derivations) are finitely generated, we can consider
other invariant tensors. For instance, invariant horizontal 1-forms form a module over the algebra
$\A$ (containing $\Gamma_\infty(\a^*)\simeq\mathfrak{Der}(\E,\t)$ as a submodule).
They correspond to (rational-polynomial) sections of the limiting distribution of
the $G^k$-invariant forms $\t_k^*$ over $\E^k$.
This module is finitely generated in the Lie-Tresse sense, and
there is a natural pairing between the elements of this module and the module
$\mathfrak{D}(\E,\t)$ with values in $\A$.

Likewise we can prove finiteness of the module of invariant $q$-forms for any $q$
(an important example: closed invariant $(n-1)$-forms, which are invariant conservation laws).
More generally any tensorial module of fixed type, for instance $\bigl((\ot^p\t)\ot(\ot^q\t^*)\bigr){}^G$
with fixed $p,q$, is finitely generated.

What about the whole tensor algebra (all possible valencies)? Let us restrict, for instance, to
$\sum_{p,q}((\ot^p\t)\ot(\ot^q\t^*))^G$ (a combination of symmetric and skew-symmetric powers
is also possible), calling its sections over $\E''_\infty$ the $\A$-module of horizontal tensors
$\mathfrak{T}(\E,\t)$.

Is this module finitely generated in the Lie-Tresse sense?
We do not know the answer in general, but understand the generic situation.

 \begin{theorem}
Let $\E$ be ample (for instance the whole space of jets), and the number of scalar differential
invariants of the pseudogroup $G$ action be infinite. Then the module $\mathfrak{T}(\E,\t)$
of $G$-invariant horizontal tensors on $\E^\infty$
is finitely generated over the algebra $\A$ in the Lie-Tresse sense globally.
 \end{theorem}

 \begin{proof}
Indeed, by Theorem \ref{Thmn} there are $n$ invariant derivatives $e_i=\nabla_i$,
so by dualization there are $n$ invariant 1-forms $e^i=\oo^i$.
We can decompose any $(p,q)$-tensor $T\in\mathfrak{T}(\E,\t)$ through these
 $$
T=T_{i_1\dots i_q}^{j_1\dots j_p}e_{j_1}\ot\dots\ot e_{j_p}\ot e^{i_1}\ot\dots\ot e^{i_q}.
 $$
The coefficients are scalar differential invariants iff $T$ is an invariant horizontal tensor.
Thus the claim follows from the global Lie-Tresse theorem.
 \end{proof}

In a similar manner we can prove finiteness of other natural geometric objects:
lifted connections, higher order differential operators etc.

\subsection{Differential syzygies and the quotient equation}\label{S34}

An algebraic relations among the generators of the algebra $\A=\mathfrak{P}_l(\E)^G$, i.e.\ among
the basic invariants $I_i$ and derivations $\nabla_j$, is called a {\it differential syzygy\/}.
Syzygies have the form $Q(\{\nabla_{J}(I_i)\})=0$, where $J$ are multi-indices and the polynomial $Q$
depends on a finite number of arguments.

Obviously the space of such $Q$ carries the structure of a module over the algebra
of invariant differential operators $\mathfrak{Diff}(\E,\t)$ generated by\footnote{The are two
obvious choices for the invariant differentiations: derivatives $\mathfrak{D}(\E,\t)$ and derivations
$\mathfrak{Der}(\E,\t)$. We choose the latter, because the primary goal is to generate $\A$.}
the algebra $\A$ and invariant derivations $\nabla_j$.
The Maurer-Cartan relations (\ref{MCe}), applied to any generator of $\A$, are among the syzygies,
so modulo those relations we can restrict to relations with ordered multi-indices $J$.
This reduces the amount of arguments for $Q$, and also allows to write
 $$
\mathfrak{Diff}(\E,\t)=\{\sum f_J\nabla_J\},
 $$
where the summation is finite, $f_J\in\A$ and the multi-indices are ordered (a choice of the order does not
play a role here). Notice that the action of a derivation $\nabla_j$ on a differential syzygy $Q$
produces another differential syzygy with possibly a bigger number of arguments. Thus we get the
D-module\footnote{Here D means a module over the algebra $\mathfrak{Diff}(\E,\t)$.
Notice that the algebra $\mathfrak{Diff}(\E,\t)$ is generated by $\A$ and $\nabla_j$,
and so is finitely generated by $I_i$ and $\nabla_j$.}
of syzygies $\mathfrak{syz}_\infty^G$.
(obviously the ring structure of $\mathfrak{syz}_\infty^G$ is surpassed by the D-module structure,
as e.g. $Q^2$ is obtained from the syzygy $Q$ by multiplication by $Q\in\A$)

To our knowledge the first result on finiteness (and rationality) of differential syzygies
(in the case of free un-constrained pseudogroup actions) is due to \cite{OP} in the micro-local case.
We generalize it to the global context.

 \begin{theorem}\label{refSyz}
The D-module $\mathfrak{syz}_\infty^G$ of differential syzygies is finitely-generated, i.e.\ there exist
a finite number of polynomials $Q_k\in\mathfrak{syz}_\infty^G$ such that any other syzygy $Q=
\sum\square_kQ_k$ for some invariant differential operators $\square_k\in\mathfrak{Diff}(\E,\t)$.
 \end{theorem}

Before proving the theorem, let us recast it into the language of differential algebra, which
significantly simplifies the proof.

Let us change the generating set of $\A$ so
that the invariant derivations $\nabla_j$ commute. This is always possible to do in $\E''$ according to
Remark \ref{Tresse-always}: we can find a collection of differential invariants $f_1,\dots,f_s$,
in addition to $I_1,\dots,I_t$, such that $\nabla_j=\hat{\p}/\hat{\p}f_j$
(equality as for derivations, i.e.\ modulo $\a^\vee$)
for $j=1,\dots,s$ in the notations of the proof of Theorem \ref{tHmT}.

Notice that passing from a general basis of derivations $\{\nabla_j\}$ to a basis of commuting derivations
is equivalent to quotient by the Maurer-Cartan relations above, and this does not change
the finite generation property.

Consider now the jet-space $J^\infty(\R^s(x),\R^t(y))$ with coordinates $x^j$ on the base and
the jet-coordinates $y^i_J=\D_J(y^i)$ and let $\mathcal{P}$ be the algebra of polynomials on it.
This algebra is filtered by $\mathcal{P}_k$ -- the algebra of polynomials on $J^k$.
We map $\Psi:\mathcal{P}\to\A$ so: $\Psi(x^j)=f_j$, $\Psi(y^i)=I_i$ and we extend the map
as a differential homomorphism. Then $\D_{x^j}$ corresponds to $\nabla_j$ and so
$\Psi(y^i_J)=\nabla_J(I_i)$. We identify
 $$
\mathfrak{syz}_\infty^G=\op{Ker}(\Psi)\subset\mathcal{P}.
 $$

 \begin{proof}[Proof of Theorem \ref{refSyz}]
Denote $\mathcal{S}_k=\op{Ker}(\Psi)\cap\mathcal{P}_k$ and
$\mathcal{S}=\cup_k\mathcal{S}_k=\mathfrak{syz}_\infty^G$.
Because $\Psi$ is a homomorphism, commuting with derivations
(meaning $\Psi\circ\D_{x^j}=\nabla_j\Psi$), this $\mathcal{S}$ is a filtered ideal
in $\mathcal{P}$ and it is D-stable: $\D_J(\mathcal{S}_k)\subset\mathcal{S}_{k+|J|}$.
Moreover it is radical because the algebra $\A$ has no zero divisors.
In the language of \cite{Ri} $\mathcal{S}$ is an algebraic differential perfect ideal.

Thus by the Ritt-Raudenbush differential algebra version of the Hilbert basis theorem \cite{Ri}
(see also Malgrange's version of Cartan-Kuranishi theorem \cite{Ma}),
$\mathcal{S}$ is finitely generated as a D-module (now D means over the algebra of
$\mathcal{C}$-differential operators in the $(x,y)$-space):
there exist polynomials $Q_1,\dots,Q_p\in\mathcal{S}$
(defining the decomposition into prime ideals)
such that any other syzygy $Q\in\mathcal{S}$ is a linear combination
$\sum\square_kQ_k$ for some differential operators in total derivatives
$\square_k=\sum b^k_\z\D_\z$, $b^k_\z\in\mathcal{P}$.

Alternatively we could prove the claim by generating every $\mathcal{S}_k$.
On a finite level existence of a finite basis follows from the classical Hilbert's basis theorem.
Increasing $k$ above $l$ we notice that the basic syzygies can be taken affine in the
highest derivatives (similarly as $\A_k$ is generated over $\A_{k-1}$ by the invariants
affine in $k$-th order jets).
The stabilization of singularities happens similarly to Theorem \ref{trma},
and the symbolic behavior of the syzygies is controlled by the same symbolic
sequence for differential invariants (in fact by the dual) as in \S\ref{S24}
 $$
0\to\mathfrak{d}_k\to\mathfrak{d}_{k-1}\ot\t^*\to\mathfrak{d}_{k-2}\ot\Lambda^2\t^*\to\dots
 $$
Vanishing of the first cohomology $H^{k,1}(\mathfrak{d})=0$, $k\ge l$,
means that the symbols of $k$-th order differential syzygy is given by the
combination of the symbols of the syzygies of order $<k$,
and the lower order terms are then reconstructed recursively.
 \end{proof}

\smallskip

The theorem gives us an exact sequence of algebras
 $$
0\to\mathfrak{syz}_\infty^G\longrightarrow\mathcal{P}\longrightarrow\A\to0,
 $$
where $\mathcal{P}$ is the free graded differential algebra with $s$ derivations and $t$ dependent
variables. Finite generation property of the (first) syzygy module means that it is an
epimorphic image of a free graded differential algebra $\mathcal{P}'$, which is an algebra of
polynomials on another jet-space $J^\infty(\R^{s'}(x'),\R^{t'}(y'))$:
$\mathfrak{syz}_\infty^G=\Psi'(\mathcal{P}')$. Then we obtain the second syzygy module as the kernel of $\Psi'$:
 $$
0\to\mathfrak{syz}_{\infty,(2)}^G\longrightarrow\mathcal{P}'\longrightarrow\mathcal{P}\longrightarrow\A\to0.
 $$
This ideal in $\mathcal{P}'$ is again finitely generated as a D-module.
Then we can compute the third syzygy module $\mathfrak{syz}_{\infty,(3)}^G$ etc
and establish finiteness of all higher $G$-invariant syzygies (i.e.\ relations between relations etc).
The derived sequence will be a free resolution of the differential algebra $\A$.

\smallskip

We can interpret the finite generation property of Theorem \ref{refSyz} as
a $G${\it-equivariant version of the Cartan-Kuranishi theorem\/}: The module of syzygies is
a system of differential equations on differential invariants, where the dependent variables are 
$I_1,\dots,I_t$ and derivations are $\nabla_1,\dots,\nabla_s$ (in the above notations, when 
$\nabla_j=\hat{\p}/\hat{\p}f_j$, then $f_1,\dots,f_s$ play the role of independent variables). 
The second syzygies are interpreted as the compatibility conditions for these differential equations.

To see this, let us assume (for simplicity), as in the above proof, that the generators are chosen so that $\nabla_j$
commute. Then $\mathcal{S}=\mathfrak{syz}_\infty^G\subset\mathcal{P}$.
Let $J^k(\R^s(x^*),\R^t(y^*))=(J^k(\R^s(x),\R^t(y)))^*$ and $\mathcal{Q}_k$ be its subset
annihilated by $\mathcal{S}_k$. Denote
$J^\infty(\R^s(x^*),\R^t(y^*))=\lim\limits_{\to}  J^k(\R^s(x^*),\R^t(y^*))\supset
\mathcal{Q}=\lim\limits_{\to}\mathcal{Q}_k$.
We have the natural projections $\pi_{i,j}^\mathcal{Q}:\mathcal{Q}_i\to\mathcal{Q}_j$ for $i>j$,
and the image of $\pi_{i,0}$ in $J^0$ is a nonempty Zariski open set.

 \begin{theorem}
Outside a finite codimension set of singularities $\mathcal{Q}$ is a formally integrable
system of differential equations, called the quotient equation.
 \end{theorem}

 \begin{proof}
Over $\CC$ the statement is an easy corollary of Hilbert's Nullstellensatz: since there
exists 1-1 correspondence between radical polynomial ideals and affine varieties (of their zeros),
the formal integrability follows from D-closedness of the ideal $\mathcal{S}$ (recall from the former proof
that it is radical, $\mathcal{S}=\sqrt{\mathcal{S}}$). Finite generation property
of $\mathcal{S}$ yields finite codimension of the singularity set in $\mathcal{Q}$.

Over $\R$ the claim follows from Dubois' real Nullstellensatz \cite{D} (see also \cite{Ris}):
the ideal of the zero variety of the ideal $\mathcal{S}$ is the real radical $\sqrt[R]{\mathcal{S}}$.
The latter can be defined as the space of all $f\in\mathcal{P}$ such that $f^m(1+\sum k_iu_i^2)\in\mathcal{S}$
for some integer $m>0$, reals $k_i>0$ and functions $u_i\in\mathcal{P}$.
Applying $\Psi$ we get $\Psi(f)(1+\sum k_i\Psi(u_i)^2)=0$, which implies $f\in\mathcal{S}$.
Thus our syzygy ideal is real radical $\mathcal{S}=\sqrt[R]{\mathcal{S}}$
and so the ideal of the family of varieties $\mathcal{Q}_k$
is D-closed. This implies $\mathcal{Q}_{k+1}\subset\mathcal{Q}_k^{(1)}$
(outside singularities) and the theorem is proved.
 \end{proof}

\smallskip

We can also consider the relations in the Lie algebra $\mathfrak{Der}(\E,\t)$
of differential syzygies. It is finitely generated by the commutator relations (Maurer-Cartan equations
for the pseudogroup) and the differential syzygies $\mathfrak{syz}_\infty^G$ .

\subsection{Asymptotic of the dimensions: Arnold's conjecture}\label{S35}

The growth of the dimensions of the space of differential invariants is
an important invariant characterizing the freedom in the equivalence problem.

Recall that the value of $\dim{\frak d}_k$ at generic point of $\E^k$
counts the number of independent differential invariants of pure order $k$.
It equals $\dim Y^k-\dim Y^{k-1}$, where the number $\dim Y^k$ is the
transcendence degree of the field $\mathfrak{R}(\E^k_a)^{G^k_a}$
(maximal number of independent differential invariants of order $\le k$).

The Hilbert function
 $$
H_G^\E(k)=\dim{\frak d}_k
 $$
is a polynomial for large $k\gg1$ by virtue of Theorem \ref{refpro}. We extend it to
the polynomial $P_G^\E(z)$, $z\in\CC$. Introducing in the same way
$H_\E(k)=\dim g_k$, $H_{G|\E}(k)=\dim\varpi_k$ ($k\gg1$),
the exact sequence of the first column (\ref{130300}) implies that
 \begin{equation}\label{PGE}
H_G^\E(z)=H_\E(z)-H_{G|\E}(z).
 \end{equation}
Moreover, following \cite{KL$_2$} we can write the Hilbert polynomial 
for differential invariants via the Lie equation for the pseudogroup $G$ and its stabilizer $\op{St}$
(the terms are defined similarly to the above):
 $$
H_G^\E(z)=H_\E(z)-H_G(z)+H_{\op{St}}(z).
 $$

Let $d_\E$ be dimension of the affine complex characteristic variety of $\E$ and $c_\E$ be its degree.
Then (\ref{PGE}) implies that for some $\tilde{d}\le d_\E$ (and $\tilde{c}\le c_\E$ if $\tilde{d}=d_\E$)
 $$
H_G^\E(z)=(c_\E\,z^{d_\E}+\dots)-(\tilde{c}\,z^{\tilde{d}}+\dots)
 $$
and in particular $H_G^\E$ is a polynomial of degree $d\le d_\E$ (the latter estimate explains why
we can restrict to $s\le d_\E$ invariant derivations).

This instantly implies rationality of the Poincar\'e function\footnote{Notice that the Hilbert and Poincar\'e
functions are related by $H^\E_G(k)=\frac1{k!}\left.\frac{d^k}{dz^k}\right|_{z=0}P^\E_G(z)$.}
that counts moduli of the equivalence problem (the series converges for $|z|<1$)
 $$
P_G^\E(z)= \sum_{k=0}^\infty \dim{\frak d}_k\cdot z^k,
 $$
answering in affirmative the question of V.\,Arnold \cite[Problem 1994-24]{A}
(cf. \cite{Sar} for the un-constrained micro-local case) on the open stratum of $\E$:

 \begin{theorem}\label{thm15}
For a transitive action of an algebraic pseudogroup $G$ on $\E$,
the Poincar\'e function $P_G^\E(z)$ is equal to $\dfrac{R(z)}{(1-z)^{d+1}}$ on $\E''_\infty$
for some polynomial $R(z)$ and the same number $d$ as above. \qed
 \end{theorem}

Arnold asked in general at which points $a_\infty\in\E^\infty$
the Poincar\'e series for the moduli\footnote{Some people interpret these as the
micro-local differential invariants, but we believe that
the moduli are meant to be the actual global invariants.}
of the problem $\sum_{k=0}^\infty \dim{\frak d}_k(a_\infty)\cdot z^k$ is rational.
This function on $\E^\infty$ with values in formal power series
depends discontinuously on $a_\infty$, and the generic stratum $\E''_\infty$ is where it is
continuous (constant in $a_\infty$ but not in $z$). As demonstrated in the last example of Section \ref{S4} the behavior
at the singular strata can be much more complicated.

 \begin{remark}
We can also study growth of the $G$-moduli for other geometric objects:
tensors, differential operators etc. Similarly, the Poincar\'e series, corresponding to
any of these structures, is rational on the generic stratum.
 \end{remark}

Notice that the Poincar\'e functions of the equation
$P_\E(z)= \sum_{k=0}^\infty \dim g_k\cdot z^k$,
and the orbit spaces $P_{G|\E}(z)= \sum_{k=0}^\infty \dim\varpi_k\cdot z^k$
are related to the Poincar\'e function of the moduli: $P_G^\E(z)=P_\E(z)-P_{G|\E}(z)$.

\medskip

{\bf Example from Section \ref{S25} revisited.} In this example with the pseudogroup
$G=\op{Diff}_\text{loc}(\R)\ltimes C^\infty(\R)$ acting on $M=\R^3$ we have constructed the differential invariants
and invariant derivations. In this case the singularities are stabilized on the jet-level $l=2$,
$S_\infty=\pi_{\infty,2}^{-1}(S_2)$, but the differential invariants can be taken polynomial in jets of
order $>0$, so that $\A=\mathfrak{P}_0(\E)^G$ (in addition the invariants are rational on the fiber of the
map $J^0\pi=\R^3(x,y,u)\to\R^2(x,y)$).

The algebra $\mathcal{A}$ of differential invariants on $\E=J^\infty\pi$ is generated by $I_1$ and $\nabla_1,\nabla_2$.
The differential syzygies among them are generated by the relations
 $$
\nabla_1I_1=0,\ [\nabla_1,\nabla_2]=\frac{I_{3a}}{I_2}\nabla_2-\frac{I_{3b}}{I_2}\nabla_1
 $$
(these in turn imply $[\nabla_1,\nabla_2]I_2=0\Leftrightarrow \nabla_2I_{3a}=I_{4b}=\nabla_1I_{3b}$).

The transcendence degree of the field of rational differential invariants of order $\leq k$ is equal to 
$1+1+2+3+\dots+(k-1)=1+{k\choose2}$ (this is just the number of independent differential invariants; we have exactly $(k-1)$ of those of pure order $k$), and so the Poincar\'e function is
 \[
P^\E_G(z)=z+z^2+2z^3+3z^4+\dots=\frac{z(z^2-z+1)}{(z-1)^2}.
 \]

To obtain the quotient equation of \S\ref{S34} we pass from the invariant derivations $(\nabla_1,\nabla_2)$
to the Tresse derivatives associated to the invariants $(I_1,I_2)$:
 $$
\frac{D}{DI_1}=-\frac{I_{3b}}{I_2I_{3a}}\nabla_1+\frac1{I_2}\nabla_2,\quad
\frac{D}{DI_2}=\frac1{I_{3a}}\nabla_1
 $$
(the coefficients are found by computing the action on $I_1,I_2$).

Denote $I_1=\xi$, $I_2=\eta$, $I_{3a}=v$, $I_{3b}=w$, so that $\frac{D}{DI_1}=\partial_\xi$,
$\frac{D}{DI_2}=\partial_\eta$. Applying the above Tresse derivatives to $v,w$ we get 4 equations
involving $I_{4a},I_{4b},I_{4c}$. Excluding these 3 parameters we obtain the
(undetermined) differential equation
 $$
\eta\,v_\xi=v\,w_\eta-w\,v_\eta.
 $$
It is point equivalent (by the transformation $U=\frac{-1}v$, $V=\frac{w}v$, $X=\xi$, $Y=\frac12\eta^2$)
to the differential equation
 $$
U_X=V_Y,
 $$
and this is our quotient equation $\mathcal{Q}$.
Its Poincar\'e function is equal to
 $$
P_\mathcal{Q}(z)=2+3z+4z^2+5z^3+\dots=\frac{2-z}{(1-z)^2}.
 $$

\section{Examples and counter-examples}\label{S4}

In this section we give new examples illustrating importance of our assumptions,
and discuss how to calculate the algebra of scalar differential invariants.

\subsection{On calculation of differential invariants}\label{S41}

There are two approaches to calculate differential invariants of a pseudogroup action.
Both are essentially microlocal, as one uses the Lie algebra sheaf while the other
the local pseudogroup (a germ of unity).

The first method is related to solution of the Lie equation $L_{\hat X}(I)=0$, where
$X\in\mathcal{G}$ is the element of the Lie algebra sheaf of $G$ and $I$ a function on $\E$.
This is a linear PDE, and in some cases with additional symmetry it can be effectively solved
for lower order invariants. The equation for invariant derivations is also linear, so if
these two steps are resolved, the Lie-Tresse theorem will yield the algebra $\A$.
But in general to find the solutions is a difficult task.

The other approach is to use parametrization of the pseudogroup (the method of moving frames)
and to exclude the parameters along the orbit, to obtain the expressions of differential
invariants. The calculations here are more algorithmic as the action is affine in the
highest derivatives, so provided the arising algebraic equations can be solved, the algebra
$\A$ can be computed. However there are restrictions of the method: the group $G$ shall be
nicely parametrized
and the action has to be locally free (the standard technical assumption).

To calculate the global rational-polynomial invariants any combination of these methods
can be applied, with an essential input coming from the constructive invariant theory. Indeed,
the action of $G$ prolonged to the $k$-jets reduces to an algebraic action of the Lie group
$G_a^k$ on the algebraic manifold $\E_a^k$. One searches for rational-polynomial
invariants of this action.

This is still a difficult but much more feasible task, because due to Theorem \ref{Thm2} we are looking
for invariants in the smaller space of rational-polynomial differential functions. In the first approach
the rational solutions of linear PDEs can be found by applications of the methods of differential algebra
and differential Galois theory. In the second approach the Gr\"obner basis methods are useful.

 \begin{remark}
The method of constructing differential invariants from a cross-section became popular
recently \cite{OP}. Globally it is restrictive as not all pseudogroups admit such sections
(whence the assumption of freeness of the action). Even on the level of finite jets,
for the action of algebraic Lie groups the section may not exist. However there always exists a
quasi-section (an irreducible subvariety intersecting a generic orbit in finite number of points),
see \cite[Proposition 2.7]{PV}. These can be used for constructive approach to differential invariants.
 \end{remark}

Let us notice that, according to the approach of Sections \ref{S23}, \ref{S33},
the invariant derivations can be also constructed via the invariant theory.
The construction suggests that invariant derivations can appear
lower in jet-level than a set of differential invariants with independent
horizontal differentials (yielding the Tresse derivatives).
This is actually observed in computations \cite{OP,KL$_4$,Man}.


\smallskip

Yet another scenario occurs. Sometimes it is difficult to compute all generators of the algebra $\A$,
but one can find some set of differential invariants and a basis of invariant derivations
(for instance, as Tresse derivatives of the first $n$ invariants). If the separation property
holds for the obtained collection, it can be used to generate the entire algebra of differential invariants.

 \begin{theorem}
Let rational differential invariants $\{I_\a\}_{\a\in A}$ separate the $G$-orbits in $\E^\infty$.
Then they generate (by the usual algebraic operations) the entire algebra $\mathcal{A}$
of differential invariants.
 \end{theorem}

 \begin{proof}
Let $\tilde{\mathcal{A}}_k\subset\mathcal{A}_k$ be the subalgebra generated (as indicated) by
$\{I_j\}$ in jet-order $\le k$. This is a subalgebra of usual invariant functions
$\mathcal{O}(\E''_k)^G$ separating generic $G$-orbits, so by the categorical property
(see Lemma 2.1 \cite{PV}) it generates $\mathcal{A}_k$.
The claim follows because $k$ is arbitrary.
 \end{proof}

\smallskip

This statement is useful in applications, when a priori some high order differential invariants
separating the generic $G$-orbits are known (for instance on the basis of the
principle of $n$ invariants \cite{ALV}). Then also the lower order invariants
will be obtained via syzygies of the initial differential invariants
(even though invariant differentiations raise the order of invariants, their combinations can result in a
differential invariant of a lower order).

\subsection{Examples of calculations}\label{S42}

As we mentioned in the Introduction most of the classical examples are
related to calculation of micro-local differential invariants. To establish equivalence
it is necessary that they are preserved, but it is not always sufficient.

The first non-trivial equivalence result involving an infinite pseudogroup
was the classification of 2nd order ODEs with respect to point transformations.
This was initiated by S.\,Lie and R.\,Liouville and essentially finished by A.\,Tresse \cite{Tr$_2$},
see an overview in \cite{Kr}. In the latter reference the invariants are written
via rational generators, and this implies solution of the global equivalence problem
(for non-singular ordinary differential equations of order 2).

Differential invariants for some other problems involve algebraic roots, and
so can be re-written via rational generators. For instance the celebrated E.\,Cartan's 
5-variables paper \cite[p.170]{C$_2$} contains a relative invariant
$I$ such that $I^4$ is a bona fide differential invariant
(this is how it is claimed in loc.cited, but in fact already $I^2$ is an invariant \cite{AK}).
More on this will be said in the next section.

In some purely algebraic problems differential invariants turn out useful, see \cite{Ol$_2$}.
It is important to use global invariants (as algebraic classifications are always global).
On this way the classical equivalence problem for binary and ternary forms was solved
in \cite{BL}, and a more general problem on equivalence under an irreducible algebraic action
of a reductive Lie group can be also solved.

Let us discuss some other examples illustrating our conditions and methods.

\smallskip

 {\bf Example.} Consider the flex equation $\E$ that appeared in \cite{GL}:
 $$
u_y^2u_{xx}-2u_xu_yu_{xy}+u_x^2u_{yy}=0.
 $$
The group $G=\op{SL}_3\times\op{Diff}(\R)$ acts on $\E$ by symmetries, where the first factor
corresponds to projective transformations of $B=\R^2(x,y)$ and the second factor
corresponds to invertible changes of the dependent variable $u\mapsto U(u)$.

Here the characteristic variety at the point $a_1=(x,y,u,u_x,u_y)$ (since the equation is
quasi-linear $\op{Char}^\CC(\E,a_2)\subset{}^\CC\mathbb{P}T^*B$ does not depend
on a choice of $a_2\in F(a_1)$) is given by the linear equation
 $$
\op{Char}^\CC(\E,a_1)=\{[p_x:p_y]\,|\,u_yp_x=u_xp_y\}\subset \mathbb{CP}^1.
 $$
However even though this variety belongs to a hypersurface (in our case: a point),
there is no such horizontal direction at $a_0=(x,y,u)$ that annihilates all
characteristics at all points $a_1\in F(a_0)$. And this implies existence of 2 invariant
derivatives.
To simplify the argument let us change $G$ to its normal subgroup (second component) $G_2=\op{Diff}(\R)$;
the differential invariants of $G$ are then $G_1$-invariant differential invariants of $G_2$.

For the pseudogroup $G_2$ we have: the invariant derivatives are $\D_x,\D_y$ and
the basic differential invariant is $w=u_x/u_y$
with the only relation (differential syzygy which is the flex equation written via $w$) -- the
equation of gas dynamics
 \begin{equation}\label{gasdyn}
w_x=w\,w_y.
 \end{equation}
To include the first component $G_1=G/G_2=\op{SL}_3$ we re-write the Lie algebra of the first component
via $w$. Here are the generators of $\g_1=\op{Lie}(G_1)$:
 \begin{multline*}
\p_x,\ \p_y,\ x\,\p_x-w\,\p_w,\ y\,\p_y+w\,\p_w,\ x\,\p_y-\p_w,\
y\,\p_x+w^2\p_w,\\
x^2\p_x+xy\,\p_y-(xw+y)\,\p_w,\ xy\,\p_x+y^2\p_y+(xw^2+yw)\,\p_w.
 \end{multline*}
Thus we have an action of a finite-dimensional Lie group on the equation (\ref{gasdyn})
and the validity of the Lie-Tresse theorem in such situation is known.

For completeness let us provide the exact formulae for differential invariants. They are
micro-locally generated by
 $$
I_6=\frac{9w_2^3w_6-200w_3^4-72w_2^2w_3w_5+300w_2w_3^2w_4-45w_2^2w_4^2}{R^{4/3}}
 $$
and
 $$
\nabla_1=\frac{6w_2}{R^{1/3}}\D_y-\frac{w_2w_4-(4/3)w_3^2}{w_2^2R^{1/3}}(\D_x-w\D_y)
 $$
where $R=9w_2^2w_5-45w_2w_3w_4+40w_3^3$ and $w_i=\D_y^i(w)$.

The module of invariant derivatives $\mathfrak{D}(\E,\t)$ is generated by $\nabla_1$ and
 $$
\nabla_2=\frac{R^{1/3}}{w_2^2}(\D_x-w\D_y)
 $$
The last derivative belongs to $\a_H^\vee$ and so is trivial as a derivation,
i.e.\ it does not generate new invariants (it is a Cauchy characteristic), so that
$\mathfrak{Der}=\langle\nabla_1\rangle$.
It is interesting to note that $\nabla_1(\op{mod}\nabla_2)$ coincides with the classical
Study derivative from the theory of curves on the projective plane, see \cite{St} and \cite[\S42]{Kl}.
The global rational differential invariants are generated by $I_6^3$ and $I_6^{-1}\nabla_1$.


\smallskip

 {\bf Example.} Consider the natural action of the pseudogroup $G=\op{Diff}_\text{loc}(\R)\times\R\times\R$
on $\R^3=\R^2(x,y)\times\R^1(u)$. The corresponding Lie algebra sheaf $\mathcal{G}$ is transitive and has generators
 $$
f(x)\,\p_x,\ \p_y, \p_u.
 $$
For the action of $G$ on un-constrained jets $J^\infty(\R^2)$ there are two invariant derivatives
$\frac1{u_x}\D_x,\D_y$ and the basic differential invariant $u_y$, so that
the algebra of differential invariants
$\mathfrak{P}_1^G=\langle u_y,u_{xy}/u_x,u_{yy},\dots\rangle$ is generated by them.

The differential equation $\E=\{u_x=0\}\subset J^\infty(\R^2)$ is $G$-invariant, and
the global Lie-Tresse theorem applies.
But for the action of $G$ on $\E$ there is only one invariant derivative
$\D_y$ (together with its multiples by differential invariants).
The algebra of differential invariants $\A=\mathfrak{P}_1(\E)^G=\langle u_y,u_{yy},u_{yyy},\dots\rangle$
is freely generated by $\D_y$ and $u_y$.

 \begin{remark}
This example shows that the number of invariant derivatives needs not to be $n=\dim\t$.
 \end{remark}
On the other hand there can be more invariant derivatives than one can expect from the characteristic 
variety of $\E$, and as in the previous example we get $\mathfrak{D}(\E,\t)\supsetneq\mathfrak{Der}(\E,\t)$.
Indeed, if we consider the subgroup $\R^3\subset G$ of translations, then $\D_x$
is an invariant derivative on $\E$, but it acts trivially on $\mathcal{A}$.

\smallskip

 {\bf Example.} Consider now a bigger pseudogroup $G$ on $\R^3=\R^2(x,y)\times\R^1(u)$
with the transitive Lie algebra sheaf $\mathcal{G}$ given by
 $$
f(x,y)\,\p_x,\ \p_y, \p_u.
 $$
This algebra acts transitively in the complement to the equation $\E=\{u_x=0\}$,
so there are no differential invariants. There are however invariant horizontal fields
in accordance with Theorem \ref{Thmn} (of course it is not possible to find them neither
as derivations, as the latter are trivial, nor as Tresse derivatives). They form a 2-dimensional
commutative Lie algebra (indeed, any third field should be a linear combination of these
with coefficients being differential invariants)
 $$
\mathfrak{D}(J^\infty,\t)=\Bigl\langle\frac1{u_x}\D_x,\D_y-\frac{u_y}{u_x}\D_x\Bigr\rangle.
 $$
Restriction to $\E$ yields non-trivial algebra of differential invariants
$\A=\mathfrak{P}_1(\E)^G=\langle u_y,u_{yy},u_{yyy},\dots\rangle$, it is generated almost in
the same way as above, except that we shall take $[\D_y]=\D_y\mod\langle\D_x\rangle$ as an invariant
derivation. This is the only invariant derivation up to multiplication by a differential invariant,
$\mathfrak{Der}=\langle[\D_y]\rangle$.

But it is invariantly defined only as the equivalence class -- there are no invariant
horizontal vector fields at all $\mathfrak{D}(\E,\t)=0$ !

\smallskip

 {\bf Example.} The Liouville equation $u_{xy}=e^u$ is automorphic: the pseudogroup
of symmetries $G=\op{Diff}_\text{loc}(\R)\times\op{Diff}_\text{loc}(\R)$ generated by
the transformations $x\mapsto X(x)$, $y\mapsto Y(y)$ acts transitively on the space of solutions
(the corresponding Lie algebra sheaf embedded into $\mathfrak{D}(J^0)$, $J^0=\R^2(x,y)\times\R(u)$
is generated by $a(x)\p_x-a'(x)\p_u$, $b(y)\p_y-b'(y)\p_u$ for arbitrary functions $a(x),b(y)$).
This means that there are no differential invariants: $\mathcal{A}=\R$.

It is easy to check that there are no invariant derivatives (horizontal vector fields) as well.
This is in agreement with Theorem \ref{Thmn} since its both assumptions fail here: $\mathcal{A}$ is
not infinite-dimensional and the Liouville equation is not ample.

\subsection{Roots in differential invariants}\label{S43}

By our main theorem we can use only rational functions in description of
global differential invariants. However books are full of examples
of differential invariants with algebraic roots, see e.g. \cite{Th,Ol$_1$,Ol$_2$,KJ,KL$_4$,Man}.

This visible contradiction can be resolved in two ways. First of all many of the
wide-spread invariants with roots are not invariants in the global sense, but only
micro-local invariants. Equivalently they are invariants of the Lie algebra (satisfy
the linear PDE $L_X(I)=0$) or invariants of the local Lie group (where the size of the
neighborhood of unity in $U\subset G$, for which $g^*I(a_\infty)=I(a_\infty)$ $\forall g\in U$,
depends on the point $a_\infty\in\E^\infty$). See the example
from the Introduction.

Another possibility is that the roots are in proper place, but the classification deals
with coverings in global problem like orientations or spin. Action of quotients of
algebraic groups which are not themselves algebraic can lead to similar problems.
Let us explain appearance of roots in two examples.

\smallskip

 {\bf Example.}
Consider the action of the proper motion group $E(2)_+=SO(2)\ltimes\R^2$ on the space of oriented
curves in $M=\R^2$. The action lifts to the space of jets $J^\infty(M,1)$.
We can chose local coordinates on $M$ such that $dx$ gives positive orientation and $dy$
positive co-orientation on the graph-curves $y=y(x)$. These induce canonical coordinates
$(x,y,y_1,y_2,\dots)$ in the open chart $J^\infty(\R,\R)\subset J^\infty(M,1)$,
and we write the curvature and the invariant derivation
 \begin{equation}\label{SO}
\kappa=\dfrac{y_2}{(1+y_1^2)^{3/2}},\qquad\ \dfrac{d}{ds}=\dfrac1{\sqrt{1+y_1^2}}\dfrac{d}{dx}.
 \end{equation}
Both are invariant with respect to $E(2)_+$ and they generate the algebra of differential invariants.


If the curves are not orientable or we act by the entire motion group $E(2)=O(2)\ltimes\R^2$
(both conditions imply that the curves are not co-orientable), then $\kappa$ is not an invariant
and we shall change the above pair of generators to
 \begin{equation}\label{O}
\kappa^2=\dfrac{y_2^2}{(1+y_1^2)^3},\qquad\ \nabla=\kappa\dfrac{d}{ds}=\dfrac{y_2}{(1+y_1^2)^2}\dfrac{d}{dx}.
 \end{equation}

It seem that (\ref{O}) satisfies but (\ref{SO}) contradicts our version of Lie-Tresse theorem.
Of course, (\ref{O}) are also rational invariants of $E(2)_+$ action, but they do not distinguish
orbits (for instance, the upper and lower half-circles both oriented from left to right). What is the problem?

To understand this consider the stabilizer $S^1=SO(2)\subset E(2)_+$ of the point $(0,0)\in M$.
It acts on the fiber of the projection $J^2\to M$, which is $S^1\times\R$ in the case of oriented curves
and is $\R P^1\times\R$ in the case of non-oriented curves. The action is $(\phi,h)\mapsto(\phi+t,h)$
and so $h\in\R$ is (a rational) invariant.

However this is the case of the jet-space $J^2(M,1)$, while in the affine chart the action writes
$(y_1,y_2)\mapsto(\frac{y_1\cos t-\sin t}{y_1\sin t+\cos t},\frac{y_2}{(y_1\sin t+\cos t)^3})$,
$t\in S^1$. The affine chart consists of one piece $\R P^1\setminus\{1\}$ in the non-oriented case,
but it has two pieces $S^1\setminus\{\pm1\}$ in the oriented case. Since the two half-circles
(= two lines) constitute a reducible variety, there is no contradiction with our Lie-Tresse theorem.
In fact on one component (that's to say e.g. $dx$ gives the orientation)
we can use the rational generators (\ref{O}) and they do separate orbits.

\smallskip

 {\bf Example.} Another occurrence of roots is the classical algebraic problem of characterizing
quadrics on the plane:
 \begin{equation}\label{quadriC}
u=ax^2+2bxy+cy^2+2dx+2ey+f=0.
 \end{equation}
The group $G=E(2)\times\R^*$ acts on the plane $(x,y)$ by the first component, and on the function $u$
(quadric) by the second (rescalings).
The invariants are (in \cite{PV} the cubic roots of these are given, but then they are ill-defined
over $\CC$)
 $$
I_1=\frac{(ac-b^2)^3}{\Delta^2},\ I_2=\frac{(a+c)^3}\Delta,\qquad \text{ where }\
\Delta=\begin{vmatrix}a & b & d \\ b & c & e \\ d & e & f\end{vmatrix}.
 $$
We can obtain the invariants also through the action of $G$ on the Monge equation $\E$
characterizing quadrics (as functions $y=y(x)$)
 $$
V=y_5y_2^2-5y_2y_3y_4+\tfrac{40}9{y_3}^3=0.
 $$
 \begin{remark}
This $V$ together with $U=y_2$ are the basic projective relative differential invariants of the curves
on the projective plane. The basic absolute invariant $R\cdot V^{-8/3}$ (with $R=U^4V\cdot y_7+\dots$
a differential polynomial of order 7 \cite{H}, see also \cite{KL$_4$})
in this case contains the cubic root, but can be changed to the global rational differential invariant
$R^3/V^8$.
 \end{remark}

Indeed, the action has 3 (micro-local) differential invariants (they contain roots
-- which can be eliminates as above -- but will be of temporal use)
 $$
\kappa=\frac{y_2}{(1+y_1^2)^{3/2}},\
\kappa'=\frac{d}{ds}\kappa,\
\kappa''=\frac{d^2}{ds^2}\kappa;\qquad
\frac{d}{ds}=\frac1{\sqrt{1+y_1^2}}\frac{d}{dx}.
 $$
They satisfy the following differential syzygy (our equation $\E$)
 $$
\kappa'''=5\,\frac{\kappa'\kappa''}{\kappa}-\frac{40}9\frac{\kappa'^3}{\kappa^2}-4\kappa^2\kappa'
 $$
with the solution given implicitly $s=\int(C_1\kappa^{8/3}+C_2\kappa^{10/3}-9\kappa^4)^{-1/2}d\kappa$.
Extracting constants (first integrals) from this expression we get two invariants of the
solution space $\op{Sol}(\E)$:
 $$
j_1=\frac{3\kappa\kappa''-5\kappa'^2+9\kappa^4}{\kappa^{8/3}},\quad
j_2=\frac{3\kappa\kappa''-4\kappa'^2+18\kappa^4}{\kappa^{10/3}}.
 $$
Expressed in jet-coordinates they equal
 $$
j_1=\frac{3y_2y_4-5y_3^2}{y_2^{8/3}},\quad j_2=
\frac{3y_1^2y_2y_4-4y_1^2y_3^2-6y_1y_2^2y_3+9y_2^4+3y_2y_4-4y_3^2}{y_2^{10/3}}.
 $$
The bona fide rational invariants of this algebraic problem are $J_1=j_1^3$, $J_2=j_2^3$
(they separate orbits and correspond to the previous invariants $I_1$, $I_2$,
when $y$ is expressed through $x$ from (\ref{quadriC})),
while roots in the above formulae are the results of intermediate calculations.

\subsection{Non-algebraic situation}\label{S44}

If we drop the requirement of algebraic action, some of the results continue to hold.
Namely the prolonged action is affine and so algebraic in higher jets. Thus the
finite-generation property will hold over a neighborhood in finite jets, where we have
the property for the corresponding Lie group action.

However the separation property for the orbits can fail, so that the algebra of differential
invariants will distinguish only between the closures of the orbits.

\smallskip

 {\bf Example.}
Consider the action of the 5-dimensional Lie group $G=\R^4\rtimes\R^1$ on 
$M=\R^4(x^1,x^2,x^3,x^4)$ with the Lie algebra $\g=\langle\p_{x^1},\dots,\p_{x^4},\xi\rangle$, where $\xi=(x^2\p_{x^1}-x^1\p_{x^2})-\ll(x^4\p_{x^3}-x^3\p_{x^4})$
and $\ll$ is a generic irrational number.

The action is transitive and the stabilizer group $G_x$ is one-dimensional. We embed $G$ into the diffeomorphism
group of $\R^5(x^1,x^2,x^3,x^4,u)$ or equivalently let $G$ act on the space $J^\infty(M)$ of jets
of functions on $M$. The space $J_x^k(M)$ has coordinates $u_\z$ with the multi-index $\z$
of length $|\z|\le k$.

The first of them $u=u_\emptyset$ is obviously an invariant\footnote{This makes the group action intransitive,
but it is irrelevant for this example. Simply add the 6th generator $\p_u$ to $\g$, or observe that
all our constructions work without this.} of $G$. On the jet-level $k=1$
we have 3 more differential invariants: $I_1'=u_1^2+u_2^2$, $I_1''=u_3^2+u_4^2$ and $I_1'''=\arctan(u_4/u_3)+\ll\arctan(u_2/u_1)$.
The last invariant $I_1'''$ is however micro-local and has to be omitted in the list of global invariants.

Geometric explanation of the absence of one invariant in the global sense is the following:
$T^*=\R^4=\R^2(u_1,u_2)\times\R^2(u_3,u_4)$ has an invariant foliation by tori which are
the product of concentric circles in the factors, and the orbit of $\xi$ on every torus is
an irrational winding with the slope $\ll$, whence the closure of almost every orbit is 2-dimensional.

Thus already on this step we observe that the algebra of micro-local invariants differs from
the algebra of global differential invariants. Also notice that the omitted non-global
invariant $I_1'''$ is not a rational function and this is in agreement with the fact that
the field $\xi\in\g$ is not a replica in the sense of Chevalley \cite{Ch} and $G$ is not algebraic.

Consider the prolongations $\xi^{(k)}$ of the last vector field of $\g$. This field represent the action of $G$
in $J^k_x=\oplus_0^k S^kT^*$, where $T=T_xM$. It is semi-simple, and every summand is invariant.
Restriction to $T$ has the purely imaginary spectrum $\op{Sp}(\xi^{(1)}|T)=\{\pm i,\pm\ll\,i\}$, and consequently the spectrum on
$S^jT$ is the $j$-multiple Minkowski sum of $\op{Sp}(\xi^{(1)}|T)$ with itself (elements in the sum enter with multiplicity):
$\op{Sp}(\xi^{(k)}|J^k_x)=\{\pm s\,i\pm \ll\,t\,i\,|\,0\le s,t\le s+t\le k\}$ (some elements enter with multiplicity).

In complex coordinates $z=x_1+ix_2$, $w=x_3+ix_4$ the first order invariants are $I_1'=|u_z|^2$ and $I_1''=|u_w|^2$.
The higher order invariants are
 $$
I_{pqrs}=\frac1{u_z^pu_{\bar z}^qu_w^ru_{\bar w}^s}\frac{\p^{p+q+r+s}u}{\p z^p\p\bar z^q\p w^r\p\bar w^s},\quad
p+q+r+s\ge2
 $$
(to get real invariants one has to take the real and complex parts unless $p=q$, $r=s$; also we can let $p\le q$
and $r\le s$ if $p=q$). The algebra of differential invariants $\A$ is generated by these,
and has the following invariant derivatives
(again to stay in the real category one separates the real and imaginary parts):
 $$
\nabla_z=\frac1{u_z}\D_z,\ \nabla_{\bar z}=\frac1{u_{\bar z}}\D_{\bar z},\
\nabla_w=\frac1{u_w}\D_w,\ \nabla_{\bar w}=\frac1{u_{\bar w}}\D_{\bar w}.
 $$
Then $\A$ is finitely generated by $I_1',I_1'',I_{1100},I_{1010},I_{0011}$ and
$\nabla_z,\nabla_{\bar z},\nabla_w,\nabla_{\bar w}$ (the minimal set of generators is
$I_1',I_1''$ and $\nabla_z,\nabla_{\bar z},\nabla_w,\nabla_{\bar w}$ since 6 of the 10 second order
differential invariants occur as the coefficients of the commutators of the invariant derivatives).

Thus we see that the Lie-Tresse theorem holds, though the invariants from the algebra
$\A$ do not separate the $G$-orbits.

\subsection{Singular systems}\label{S45}

In classification problems we describe the differential invariants characterizing
the regular orbits. Complement to the regular set consists of singular orbits, but can be stratified
into invariant equations $\E(\a)$, for each of which the $G$-action produces its own algebra
of differential invariants. Provided the equation is regular, irreducible and the action is transitive,
the Lie-Tresse theorem applies in our global version, and the description of this algebra is finite.

However, some singularities can be more complicated, for instance if the projection of the equation
to $M$ is not surjective or the action of $G$ is not transitive, and for them the theorem could fail.
To illustrate this we discuss an example of conservative vector fields with isolated singularities.
On the formal level the latter set is a subset of jets with fixed point
$\E\subset J^\infty_0(\R^{2n},\R^{2n})$, and so can be considered only as a generalized equation
(a subset in the space of jets that does not project onto the base).

\medskip

 {\bf Example.}
Consider the equivalence problem for Hamiltonian systems near
a non-dege\-ne\-rate linearly stable equilibrium point (the point is isolated and
we take it to be 0). In this equivalence problem the pseudogroup consists of
local diffeomorphisms of the germ of 0.

Equivalently the problem is this: the pseudogroup $G$ consists of germs of
symplectic diffeomorphisms of $(\R^{2n},\oo_0)$ preserving 0. Its subgroup acts on the
equation $\E$ consisting of the germs of functions $H$ vanishing at 0 to order 2
with the operator $A_H=\oo^{-1}d^2_0H$ having purely imaginary spectrum.
In addition, we assume that $\op{Sp}(A_H)=\{\pm i\ll_1,\dots,\pm i\ll_n\}$
has no resonances of any order.

The equation is indeed singular, as it is projected to $J^1(\R^{2n})$ to one point,
$\E\subset\pi_{2,1}^{-1}(0)$. But in addition, the condition of absence of resonances
is non-algebraic (but the closure of $\E$ is semi-algebraic).

We change, as usual, germs to jets to get the formal normal form of $H$. This is
the well-known Birkhoff normal form \cite{B}:
there are canonical coordinates (Darboux coordinates for the symplectic form $\oo_0$)
$x_1,y_1,\dots,x_n,y_n$ such that with the notations $\tau_i=x_i^2+y_i^2$,
$\tau^\z=\t_1^{\z_1}\cdots\t_n^{\z_n}$, the infinite jet of $H$ is
 $$
j_0^\infty(H)= \sum a_\z\tau^\z.
 $$
The coefficients $a_\z$ are the differential invariants of the problem.

One can write them in jet coordinates, though the formulae are rather complicated.
For instance, for $n=1$ the first invariant $a_1$ corresponds to the Hessian
 $$
I_2=H_{11}H_{22}-H_{12}^2
 $$
(we write $x^1,x^2$ instead of $x_1,y_1$ and denote derivatives by the subscripts). The next differential
invariant is
 \begin{multline*}
I_4=(3H_{22}H_{2222}-5H_{222}^2)H_{11}^3+\bigl(-3H_{12}^2H_{2222}+(30H_{122}H_{222}-12H_{22}H_{1222})H_{12}\\
+6H_{22}^2H_{1122}+(-9H_{122}^2-6H_{112}H_{222})H_{22})H_{11}^2+\Bigl(12H_{12}^3H_{1222}+\\
(-24H_{112}H_{222}+6H_{22}H_{1122}-36H_{122}^2)H_{12}^2+
\bigl(-12H_{22}^2H_{1112}+(54H_{122}H_{112}\\
+6H_{222}H_{111})H_{22}\bigr)H_{12}+3H_{22}^2(-3H_{112}^2-2H_{122}H_{111}+H_{22}H_{1111})\Bigr)H_{11}\\
-12H_{12}^4H_{1122}+(36H_{122}H_{112}+12H_{22}H_{1112}+4H_{222}H_{111})H_{12}^3-\\
3H_{22}(12H_{112}^2+H_{22}H_{1111}+8H_{122}H_{111})H_{12}^2
+30H_{22}^2H_{12}H_{112}H_{111}-5H_{22}^3H_{111}^2
 \end{multline*}
and we get 1 new invariant in any even order. Thus the Lie-Tresse derivative does not exist
in this case (otherwise it would produce invariants in any sufficiently high order).

One might get an impression that there is an invariant 2nd order differential operator
(which does the job of the usual invariant derivative in this case),
and indeed the operator
 \begin{multline*}
\Delta=H_{22}\D_1^2-2H_{12}\D_1\D_2+H_{11}\D_2^2+\\
I_2^{-1}(-H_{11}H_{122}H_{22}+H_{12}H_{11}H_{222}-H_{22}^2H_{111}+3H_{12}H_{112}H_{22}-2H_{12}^2H_{122})\D_1\\
\quad +I_2^{-1}(-2H_{12}^2H_{112}+3H_{12}H_{122}H_{11}+H_{111}H_{22}H_{12}-H_{11}H_{22}H_{112}-H_{11}^2H_{222})\D_2
 \end{multline*}
satisfies the identity
 $$
[\Delta,\hat X_F]|_\E=0
 $$
for any function $F\in C^\infty(\R^2)$ vanishing at 0 to the 2nd order ($X_F$ is the Hamiltonian field
with the generating function $F$).

 \begin{remark}\label{rkbtt}
The symbol of $\Delta$ is canonical: using $\oo_0$ it is identified with the operator $A_H$,
and further raising of the indices yields the Hessian 2-form $d_0^2H$.
 \end{remark}
However the restrictions of the fields from the pseudogroup $\hat X_F$ to $\E$
do no commute with $\Delta$, and this is the reason this operator does not map
$I_{2k}\to I_{2k+2}$.

 \begin{prop}
There does not exist a differential operator which maps the algebra of differential invariants to itself.
 \end{prop}

 \begin{proof}
Let us discuss the case of the 2nd order, the general case is similar.

The space of differential invariants of order $\le 2k$ is $\A_{2k}=C^\infty(I_2,\dots,I_{2k})$
(here we can equally well consider the spaces of smooth or rational or polynomial functions).
If the operator exists it maps $\Delta:\A_{2k}\to\A_{2k+2}$.

In particular, for $k=1$ we get the following. Let us decompose $\Delta=\Delta_2+\Delta_1+\Delta_0$,
where $\Delta_i$ are differential operators of pure order $i$ (working in coordinates such a splitting
is possible). Then $\Delta(1)=\Delta_0$, so this latter is a differential invariant, and can be omitted.
The symbol of $\Delta_2=\sum a_{ij}\D_i\D_j$ is also an invariant, and so is a multiple of the operator
from Remark \ref{rkbtt}: $a_{ij}=(-1)^{i+j}\ll\,H_{3-i,3-j}$. Then we have:
 $$
\Delta f(I_2)=f'(I_2)\Delta(I_2)+f''(I_2)\sum a_{ij}\D_i(I_2)\D_j(I_2).
 $$
This implies that the third order expression $\sum a_{ij}\D_i(I_2)\D_j(I_2)$ is a differential invariant
(in general, we shall also consider the case $\op{ord}(\ll)>3$ but this leads to the same result
by restricting to larger $k$), which is impossible unless $\ll=0$.

Thus $\Delta=\Delta_1$, but there are no invariant differential operators of order 1 because its symbol would
have been invariant and hence zero (alternatively: it would have mapped  $\A_{2k}\to\A_{2k+1}=\A_{2k}$).
 \end{proof}

\smallskip

For $n>1$ the algebra of differential invariants of order $\le k$ is
$\A_{2k}=C^\infty(I_2^{i_1},I_4^{i_1i_2},\dots,I_{2k}^{i_1\dots i_k})$ ($1\le i_j\le n$ numerate different
differential invariants in different orders), and we get a similar result. Hence we conclude:

 \begin{cor}
The space of differential invariants $\A_\infty$ of
formal Hamiltonian systems at a non-dege\-nerate linearly stable non-resonant
equilibrium point is not finitely generated, even in the generalized Lie-Tresse sense.
 \end{cor}

Let us also notice that the Poincar\'e function of this equivalence problem
 $$
P_G^\E(z)=\sum_{k=0}^\infty\binom{n+k-1}{k}z^{2k}=\frac1{(1-z^2)^n}
 $$
does not satisfy the conclusion of Theorem \ref{thm15}. This is often the case with
classification problems arising in singularity theory.

\section{Conclusion}\label{S5}

Differential invariants can be interpreted as invariant non-linear differential operators.
The problem of their description and significance was discussed by O.\,Veblen in the
International Congress of Mathematicians 1928 \cite{Ve}.

The structure of the algebra of differential invariants encodes the structure of the
pseudogroup and its action. It plays the same role in local differential geometry as the
invariant theory in the classical algebra. With this in mind we can informally summarize
the main results of the present paper as follows:
 \begin{quote}
{\it Hilbert and Rosenlicht theorems allow to treat the quotient of an algebraic variety by an algebraic
group action as an algebraic variety. Our global version of the Lie-Tresse theorem allows to treat
the quotient of an algebraic differential equation by an algebraic pseudogroup action as an
algebraic differential equation.}
 \end{quote}
The meaning of the latter is explained in Section \ref{S34}.

Calculations of differential invariants have several applications:
equivalence problem, homogeneity detection, image recognition, variational calculus,
computer vision, numerical computing, relativity theory, fluid dynamics and many more \cite{ALV,Ol$_3$}.
Cartan's equivalence method can be considered as an important approach to compute differential invariants,
but can in turn be obtained as a partial case of the general differential invariants theory \cite{Va}.

Differential invariants provide vital tools in integration of differential equations.
Traditionally related to the symmetry methods (and establishing exact solutions),
they find usage in the transformation theory (e.g.\ Laplace and Darboux transformation),
variational bi-complex (and conservation laws), higher symmetries and commuting flows (Lax pairs).
Our approach encompasses not only the groups of point symmetries, but also contact symmetries and
in some cases internal symmetries (the setup can require extension of the base manifold).

We almost ignored the computational aspect of the theory (see \S\ref{S41}). This uses
several algebraic methods (including Gr\"obner basis and differential algebra,
see \cite{HK}). In practice symbolic packages are helpful. For some calculations in
Section \ref{S4} we used \textsc{Maple} package \textsl{DifferentialGeometry} by I.\,Anderson.

Our examples were chosen simple to illustrate the power and limitation
of the main results. More substantial calculations will appear in future publications.


 \vspace{-3pt} \hspace{-20pt} {\hbox to 12cm{ \hrulefill }}
\vspace{3pt}

{\footnotesize \hspace{-10pt} Institute of Mathematics and
Statistics, University of Troms\o, Troms\o\ 90-37, Norway.

\hspace{-10pt} E-mails: \quad \textsc{boris.kruglikov@uit.no}, \quad
\textsc{valentin.lychagin@uit.no}.} \vspace{-1pt}

\end{document}